\documentclass[11pt]{amsart}

\usepackage{amsmath,amssymb,amsthm,amsfonts,mathrsfs}
\usepackage{amscd,stmaryrd}
\usepackage[usenames,dvipsnames]{color}
\usepackage{hyperref}
\definecolor{nicegreen}{RGB}{0,180,0}
\hypersetup{colorlinks=true,citecolor=nicegreen,linkcolor=red,urlcolor=blue,pdfstartview=FitH}
\usepackage[all]{xy}
\usepackage{enumerate}
\usepackage{stmaryrd}

\usepackage[divide={2.45cm,*,2.45cm}]{geometry}
\usepackage{enumerate}
\usepackage{color}
\usepackage[all]{xy}

\allowdisplaybreaks

\newtheorem{thm}{Theorem}[section]
\newtheorem{cor}[thm]{Corollary}
\newtheorem{lem}[thm]{Lemma}
\newtheorem{propn}[thm]{Proposition}
\newtheorem*{nothm}{Theorem}
\newtheorem*{nocor}{Corollary}

\theoremstyle{remark}
\newtheorem*{remark}{Remark}

\theoremstyle{definition}
\newtheorem{defn}[thm]{Definition}
\newtheorem*{defn1}{Definition}

\newcommand{\qp}{\mathbb{Q}_p}

\newcommand{\fp}{\mathbb{F}_p}
\newcommand{\fpb}{\ol{\mathbb{F}}_p}
\newcommand{\zp}{\mathbb{Z}_p}
\newcommand{\pp}{\mathfrak{p}}
\newcommand{\oo}{\mathfrak{o}}
\newcommand{\T}{\textnormal{T}}
\newcommand{\s}{\textnormal{S}}
\newcommand{\gal}{\mathcal{G}}
\newcommand{\ii}{\mathcal{I}}
\newcommand{\aff}{\textnormal{aff}}
\newcommand{\mm}{\mathfrak{m}}
\newcommand{\mmm}{\widetilde{\mathfrak{m}}}
\newcommand{\MM}{\mathfrak{M}}
\newcommand*{\longhookrightarrow}{\ensuremath{\lhook\joinrel\relbar\joinrel\rightarrow}}
\newcommand{\sub}[2]{\genfrac{}{}{0pt}{}{#1}{#2}}
\newcommand{\ol}[1]{\overline{#1}}

\newcommand{\cC}{{\mathcal{C}}}

\newcommand{\cH}{{\mathcal{H}}}
\newcommand{\cI}{{\mathcal{I}}}

\newcommand{\cR}{{\mathcal{R}}}
\newcommand{\cS}{{\mathcal{S}}}

\newcommand{\cZ}{{\mathcal{Z}}}

\newcommand{\bG}{{\mathbf{G}}}

\newcommand{\bJ}{{\mathbf{J}}}

\newcommand{\bN}{{\mathbf{N}}}

\newcommand{\bT}{{\mathbf{T}}}

\newcommand{\bZ}{{\mathbf{Z}}}

\newcommand{\bbN}{{\mathbb{N}}}

\newcommand{\bbR}{{\mathbb{R}}}

\newcommand{\bbZ}{{\mathbb{Z}}}

\begin{document}

\title{Pro-$p$-Iwahori invariants for $\textnormal{SL}_2$ and $L$-packets of Hecke modules}
\author{Karol Kozio\l}
\address{Department of Mathematics, University of Toronto, Toronto, ON M5S 2E4, Canada}
\email{karol@math.toronto.edu}

\begin{abstract}
Let $p$ be a prime number, and $F$ a nonarchimedean local field of residual characteristic $p$.  We explore the interaction between the pro-$p$-Iwahori-Hecke algebras of the group $\textnormal{GL}_n(F)$ and its derived subgroup $\textnormal{SL}_n(F)$.  Using the interplay between these two algebras, we deduce two main results.  The first is an equivalence of categories between Hecke modules in characteristic $p$ over the pro-$p$-Iwahori-Hecke algebra of $\textnormal{SL}_2(\mathbb{Q}_p)$ and smooth mod-$p$ representations of $\textnormal{SL}_2(\mathbb{Q}_p)$ generated by their pro-$p$-Iwahori-invariants.  The second is a ``numerical correspondence'' between packets of supersingular Hecke modules in characteristic $p$ over the pro-$p$-Iwahori-Hecke algebra of $\textnormal{SL}_n(F)$, and irreducible, $n$-dimensional projective Galois representations.  
\end{abstract}

\maketitle
\tableofcontents

\section{Introduction}

In recent years, there has been a great deal of interest and activity surrounding the (still nebulous) mod-$p$ version of the Local Langlands Program.  This situation is best understood for the group $\textnormal{GL}_2(\qp)$:  there exists a correspondence between isomorphism classes of semisimple mod-$p$ representations of $\textnormal{Gal}(\ol{\mathbb{Q}}_p/\qp)$ of dimension 2 and (certain) smooth, finite length, semisimple mod-$p$ representations of $\textnormal{GL}_2(\qp)$, due to Breuil (\cite{Br03}).  This correspondence is compatible with a $p$-adic version of the Local Langlands Correspondence, and, even better, both the mod-$p$ and $p$-adic versions are induced by a \emph{functor} (see \cite{Br032}, \cite{Br04}, \cite{Em10}, \cite{Ki09}, \cite{Ki10}, and especially \cite{Co10}, \cite{Pas10}).

Serious difficulties arise when one considers groups other than $\textnormal{GL}_2(\qp)$, however.  For example, Breuil and Pa\v{s}k\={u}nas have shown in \cite{BP12} that, for $F$ a nontrivial unramified extension of $\qp$, there is an infinite family of representations of $\textnormal{GL}_2(F)$ associated to a ``generic'' Galois representation.  Therefore, it is not clear what the ``shape'' of a mod-$p$ correspondence should be for a general reductive group.  Nevertheless, Breuil and Herzig have given a construction of the ``ordinary part'' of such a correspondence, and shown that it appears in certain spaces of mod-$p$ automorphic forms (\cite{BH12}).

An alternative viewpoint for examining these difficulties comes through the study of Hecke modules.  From this point onwards, we let $F$ be a nonarchimedean local field with residue field of size $q$ and characteristic $p$.  Let $G$ denote the group of $F$-rational points of a connected reductive group $\mathbf{G}$, which we assume to be split over $F$.  We let $G_\s$ be the $F$-rational points of the derived subgroup of $\mathbf{G}$, let $I(1)$ denote a pro-$p$-Iwahori subgroup of $G$, and set $I_\s(1) := I(1)\cap G_\s$.  Letting $\bullet$ represent either the empty symbol or ``$\s$,'' we define the \emph{pro-$p$-Iwahori-Hecke algebra} $\cH_\bullet$ as the convolution algebra of compactly supported, $\fpb$-valued functions on the double coset space $I_\bullet(1)\backslash G_\bullet/I_\bullet(1)$.

Taking the $I_\bullet(1)$-invariants of a smooth representation of $G_\bullet$ over $\fpb$ yields a functor
$$\mathcal{I}_\bullet:\mathfrak{Rep}_{\fpb}(G_\bullet)\longrightarrow \mathfrak{Mod}-\cH_\bullet$$
from the category of smooth, mod-$p$ representations of $G_\bullet$ to the category of right $\cH_\bullet$-modules.  When $G = \textnormal{GL}_2(\qp)$, the functor $\cI$ induces an equivalence between the subcategory of $\mathfrak{Rep}_{\fpb}(\textnormal{GL}_2(\qp))$ consisting of representations generated by their $I(1)$-invariants and $\mathfrak{Mod}-\cH$.  Thus one hopes to glean information about the category of representations by examining $\cH$-modules.  This doesn't give an equivalence of categories in general, however (cf. \cite{Oll09}).

One can also ask to what extent the classical Local Langlands Correspondence is reflected by modules over $\cH_\bullet$.  For example, when $G = \textnormal{GL}_n(F)$, results of Vign\'eras and Ollivier (\cite{Vig05} and \cite{Oll10}) show that we have an equality
\begin{equation}\label{glnnum}
\begin{gathered}
\xymatrix{
\#\left\{\txt{\textnormal{simple, supersingular}\\ $\cH$\textnormal{-modules of dimension}~$n$\\ \textnormal{with fixed action}\\ \textnormal{of a uniformizer}}\right\} = \#\left\{\txt{\textnormal{irreducible, mod-}$p$\\ \textnormal{Galois representations}\\ \textnormal{of dimension}~$n$~\textnormal{with fixed}\\ \textnormal{determinant of Frobenius}}\right\},
}
\end{gathered}
\end{equation}
where we consider all objects up to isomorphism, and where a supersingular module is an object intended to mirror a supercuspidal representation of $G$.  By recent work of Gro\ss e-Kl\"onne (\cite{GK13}), we now know that this numerical bijection is induced by a functor, at least in the case when $F = \qp$.

The goal of the present article is to analyze the situations of the previous two paragraphs for the group $\textnormal{SL}_n(F)$.  We begin with a general split $\bG$, and recall in Section \ref{algs} the presentations of the algebras $\cH$ and $\cH_\s$, due to Vign\'{e}ras.  Proposition \ref{embedding} shows that we have an injection $\cH_\s\longhookrightarrow \cH$, making $\cH$ into a free $\cH_\s$-module.  We then use a descent argument to deduce when certain universal modules corresponding to $\cH$ and $\cH_\s$ are flat (resp. faithfully flat, resp. projective) (Corollaries \ref{cflatiffcsflat} and \ref{flatcor}).  

From Section \ref{equivs} onwards, we assume $G = \textnormal{GL}_n(F)$ and $G_\s = \textnormal{SL}_n(F)$.  In Section \ref{equivs}, we examine more closely the functor $\mathcal{I}_\bullet$.  In order to accurately speak of an equivalence of categories, we must consider the full subcategory of $\mathfrak{Rep}_{\fpb}(G_\bullet)$ consisting of representations generated by their space of $I_\bullet(1)$-invariant vectors, which we denote $\mathfrak{Rep}_{\fpb}^{I_\bullet(1)}(G_\bullet)$.  We continue to denote by $\mathcal{I}_\bullet$ the functor of invariants restricted to this subcategory.  Our main result in this section is the following:

\begin{nothm}[Theorem \ref{equiv}]
 Assume $(n,p) = 1$.  Then the functor $\mathcal{I}$ induces an equivalence of categories between $\mathfrak{Rep}_{\fpb}^{I(1)}(G)$ and $\mathfrak{Mod}-\cH$ if and only if the functor $\mathcal{I}_\s$ induces an equivalence of categories between $\mathfrak{Rep}_{\fpb}^{I_\s(1)}(G_\s)$ and $\mathfrak{Mod}-\cH_\s$.  
\end{nothm}

Using results of Ollivier (\cite{Oll09}) on the functor of $I(1)$-invariants for $\textnormal{GL}_2(F)$ (and a slight extension of Theorem \ref{equiv}), we obtain the following corollary.

\begin{nocor}[Corollary \ref{equivcor}]
 The functor $\mathcal{I}_\s$ induces an equivalence between $\mathfrak{Rep}_{\fpb}^{I_\s(1)}(\textnormal{SL}_2(\qp))$ and $\mathfrak{Mod}-\cH_\s$ when $p>2$.  
\end{nocor}

As a second application of the interaction between the algebras $\cH$ and $\cH_\s$, we investigate the supersingular modules for $\cH_\s$ in Section \ref{packets}.  The precise notion of supersingularity may be found in \cite{Vig05}, and a complete description of such modules may be found in \cite{Oll12}.  There is a natural conjugation action on $\cH$ by the multiplicative subgroup of elements of length $0$, which preserves the subspace $\cH_\s$.  Therefore, given any supersingular character $\chi$ of $\cH_\s$ and an element $\T_\omega$ of $\cH$ of length $0$, we may define a new character $\T_\omega\cdot\chi$, given on $\cH_\s$ by first conjugating the argument by $\T_\omega$ and then applying $\chi$.  Mimicking the classical case of complex representations of $G_\s$, we make the following definition:

\begin{defn1}
 An \emph{L-packet of supersingular $\cH_\s$-modules} is an orbit of the subgroup of $\cH$ of length $0$ elements on the set of supersingular characters of $\cH_\s$.  
\end{defn1}

In order to proceed further, we impose an additional ``regularity'' condition on $L$-packets.  Given this, we are able to count the number of regular, supersingular $L$-packets of size $d$, for $d$ a divisor of $n$ (Corollary \ref{numorbits}).  Our next goal is to relate the $L$-packets thus constructed to projective, mod-$p$ Galois representations, which we take up in Section \ref{galois}.  Our main theorem is as follows.  

\begin{nothm}[Corollaries \ref{numprojreps} and \ref{finalcor}]
Let $d$ be a divisor of $n$.  The number of regular supersingular $L$-packets of $\cH_\s$-modules of size $d$ is equal to the number of isomorphism classes of irreducible projective Galois representations 
$$\sigma:\textnormal{Gal}(\overline{F}/F)\longrightarrow \textnormal{PGL}_n(\fpb)$$
having exactly $d\frac{q - 1}{n}$ isomorphism classes of lifts to genuine Galois representations.  This number is equal to 
$$h(d) = \frac{1}{d}\sum_{e|d}\mu\left(\frac{d}{e}\right)g(e),$$
where $\mu$ denotes the M\"obius function, and 
$$g(e) = \sum_{f|n} \mu\left(\frac{n}{f}\right)\left(\frac{f}{(e,f)}, q - 1\right)\frac{q^{(e,f)} - 1}{q - 1}.$$
In particular, the number of regular supersingular $L$-packets of $\cH_\s$-modules is equal to the number of isomorphism classes of irreducible projective Galois representations of dimension $n$.  
\end{nothm}

Some remarks are in order.  Firstly, we note the function $g(e)$ can be computed quite explicitly in terms of Euler's phi function.  With some work, one can show that $h(d)\neq 0$ implies $d\frac{q - 1}{n}\in \mathbb{Z}$ (Lemma \ref{g(d)neq0}), so that the statement of the above theorem makes sense.  Secondly, when $n = 2$ and $F = \qp$, we recover the bijection contained in work of Abdellatif (\cite{Ab11}).

Finally, let us remark that the work of \cite{GK13} shows how to construct a functor from the category of finite length modules over the pro-$p$-Iwahori-Hecke algebra of a general connected reductive group $\bG$, defined and split over $\qp$, to the category of \'etale $(\varphi^r, \Gamma_0)$-modules.  When $\bG = \mathbf{GL}_n$, this functor (along with Fontaine's equivalence of categories) induces the numerical correspondence \eqref{glnnum}.  We compute this functor for supersingular $L$-packets of $\cH_\s$-modules in Subsection \ref{gk}, and show how to define a map from the set of regular supersingular $L$-packets to the set of irreducible, projective mod-$p$ Galois representations.  Moreover, we show in Corollary \ref{Wbij} that this map realizes the numerical correspondence of Corollary \ref{finalcor}.  

\noindent\textbf{Acknowledgements.} I would like to thank Rachel Ollivier for her support and guidance throughout the course of working on this article, and for many extremely useful comments.  Several parts of this paper were written during the conference ``Modular Representation Theory of Finite and $p$-adic Groups'' at the Institute for Mathematical Sciences, National University of Singapore, and I would like to thank the institution for its support.  During the preparation of this article, support was also provided by NSF Grant DMS-0739400.

\section{Notation}\label{notation}

Fix a prime number $p$, and let $F$ be a nonarchimedean local field of residual characteristic $p$.  Denote by $\oo$ its ring of integers, and by $\pp$ the unique maximal ideal of $\oo$.  Fix a uniformizer $\varpi$ and let $k = \oo/\pp$ denote the (finite) residue field.  The field $k$ is a finite extension of $\mathbb{F}_p$ of size $q$.  We fix also a separable closure $\overline{F}$ of $F$, and let $k_{\ol{F}}$ denote its residue field.  Let $\iota:k_{\ol{F}}\stackrel{\sim}{\longrightarrow} \fpb$ denote a fixed isomorphism, and assume that every $\fpb^\times$-valued character factors through $\iota$.  

Let $\bG$ denote a connected, reductive group, with derived subgroup $\bG_\s$.  We assume that $\bG$ is split over $F$.  For any algebraic subgroup $\bJ$ of $\bG$, we denote by $\bJ_\s$ its intersection with $\bG_\s$.  We let $\bT$ denote a fixed split maximal torus of $\bG$, so that $\bT_\s$ is a maximal torus of $\bG_\s$.  %Note that, Corollary A.2.7 of Conrad--Gabber--Prasad implies that $\bT\cap \bG_\s$ is connected, hence is a maximal torus of $\bG_\s$.  
Let $\bZ$ denote the connected center of $\bG$; note that $\bZ_\s$ is not necessarily connected.  Let $\bN$ denote the normalizer of $\bT$ in $\bG$, so that $\bN_\s$ is the normalizer of $\bT_\s$ in $\bG_\s$.  We will generally denote algebraic groups by boldface letters, and their groups of $F$-rational points by the corresponding italicized Roman letter (e.g., $G = \bG(F), T = \bT(F)$, etc.).  

\section{Weyl Groups}

In order to ease notation in the following discussion (and throughout the remainder of the article), we let $\bullet$ denote either the empty symbol or $\s$.  For an algebraic group $\bJ$, we let $X^*(\bJ)$ (resp. $X_*(\bJ)$) denote the group of algebraic characters (resp. cocharacters) of $\bJ$.  We let $\Phi_\bullet\subset X^*(\bT_\bullet)$ denote the set of roots of $\bT_\bullet$ acting on $\textnormal{Lie}(\bG_\bullet)$ by conjugation.  Restriction to $\bT_\s$ gives a bijection between $\Phi$ and $\Phi_\s$ (\cite{Bo91}, Section 21.1).

Let $A_\bullet := (X_*(\bT_\bullet)/X_*(\bZ_\bullet^\circ))\otimes_\bbZ \bbR$ be the standard apartment corresponding to $\bT_\bullet$ in the (adjoint) Bruhat--Tits building $X_\bullet$ of $G_\bullet$ (see \cite{SS97} for an overview).   Since $X_*(\bT_\s)$ is of finite index in $X_*(\bT)/X_*(\bZ)$, the apartment $A_\s$ identifies canonically with $A$.  We fix a hyperspecial point $x_0\in A_\s$ and a chamber $C\subset A_\s$ containing $x_0$ (we view both in either $A_\s$ or $A$).  Since $x_0$ is hyperspecial, the set of roots $\Phi_\bullet$ identifies with the subset of affine roots which are zero on $x_0$, and we let $\Phi_\bullet^+$ denote the set of those roots which are positive on $C$ (see Section 1.9 of \cite{Ti79}).  As above, restriction to $\bT_\s$ gives a bijection between $\Phi^+$ and $\Phi_\s^+$.

Let $I_\bullet$ denote the Iwahori subgroup in $G_\bullet$ corresponding to $C$, and $I_\bullet(1)$ its pro-$p$ radical.  We have $I\cap G_\s = I_\s$ and $I(1)\cap G_\s = I_\s(1)$.  %See, for example, Haines--Rapoport.  This requires a bit of comment.  By construction of the Borovoi fundamental group (see, e.g., Borovoi's paper "Abelian Galois Cohomology..."), we easily get that $X^*(Z(\widehat{\bG_\s})) = \pi_1(\bG_\s)\longhookrightarrow \pi_1(\bG) = X^*(Z(\widehat{\bG}))$.  Let $\imath:\bG_\s\longhookrightarrow\bG$ denote the inclusion.  By functoriality of the Kottwitz homomorphism and injectivity, we get that $\ker(\kappa_{G})\cap G_\s = \ker(\kappa_{G}\circ\imath) = \ker(\pi_1(\imath)\circ\kappa_{G_\s}) = \ker(\kappa_{G_\s})$, which is enough to give the claim.
The group $N_\bullet$ acts on $A_\bullet$ by affine transformations; the group $T_\bullet\cap I_\bullet$ acts trivially.  Moreover, the group $(T_\bullet\cap I_\bullet)/(T_\bullet\cap I_\bullet(1))$ identifies with the group of $k$-points of a torus, which we denote $T_\bullet(k)$.  The Iwahori decomposition implies $I_\bullet = T_\bullet(k)\ltimes I_\bullet(1)$ (cf. \emph{loc. cit.}, Section 3.7).

We define the following (Iwahori--)Weyl groups:
\begin{eqnarray*}
W_{0,\bullet} & := & N_\bullet/T_\bullet\\
W_\bullet & := & N_\bullet/(T_\bullet \cap I_\bullet)\\
W_\bullet(1) & := & N_\bullet/(T_\bullet \cap I_\bullet(1))
\end{eqnarray*}
Theorem 21.2 in \cite{Bo91} implies $W_{0,\bullet} \cong \bN_\bullet/\bT_\bullet$, and by the discussion in Section 21.1 of \emph{loc. cit.} we have $W_0 \cong W_{0,\s}$.   Section 3.3 of \cite{Ti79} shows that we have the following Bruhat decomposition:
$$G_\bullet = \bigsqcup_{w\in W_\bullet} I_\bullet wI_\bullet.$$
Here $I_\bullet wI_\bullet$ denotes the double coset $I_\bullet n_w I_\bullet$ for any lift $n_w$ in $N_\bullet$ of $w$.  Using this, one easily obtains the following double coset decomposition (\cite{Vig05}, Theorem 6):
\begin{equation}\label{bruhat}
G_\bullet = \bigsqcup_{w\in W_\bullet(1)} I_\bullet(1) wI_\bullet(1).
\end{equation}

The affine Weyl group $W_{\aff,\bullet}$ is defined as the subgroup of $W_\bullet$ generated by the reflections in the hyperplanes corresponding to the affine roots of $\bT_\bullet$.  We let $S_\bullet$ denote the set of reflections in the hyperplanes containing a facet of $C$; the pair $(W_{\aff,\bullet}, S_\bullet)$ is then a Coxeter system (Th\'eor\`eme 1 of \cite{Bo81} V \S 3.2).  We denote by $\ell:W_{\aff,\bullet}\longrightarrow \bbN$ the length function on $W_{\aff,\bullet}$ with respect to $S_\bullet$.  Sections 1.4 and 1.5 of \cite{Lu89} imply that the length function inflates to $W_\bullet$, and we have a decomposition
$$W_\bullet \cong \Omega_\bullet\ltimes W_{\aff,\bullet},$$
where $\Omega_\bullet$ denotes the elements of length 0 in $W_\bullet$.  Alternatively, $\Omega_\bullet$ is the subgroup of elements $\omega\in W_\bullet$ for which $\omega.C = C$.  The length function further inflates to $W_\bullet(1)$, being trivial on $(T_\bullet\cap I_\bullet)/(T_\bullet\cap I_\bullet(1))$.  

The injection $N_\s\longhookrightarrow N$ induces an injection $W_\s\longhookrightarrow W$, which fits into a diagram 
\begin{equation}\label{diag}
\begin{gathered}
\xymatrix{
1 \ar[r] & X_*(\bT_\s) \ar[r]\ar@{^{(}->}[d] & W_\s \ar[r]\ar@{^{(}->}[d] & W_{0,\s} \ar[r]\ar[d] & 1\\
1 \ar[r] & X_*(\bT) \ar[r] & W \ar[r] & W_0 \ar[r] & 1
}
\end{gathered}
\end{equation}
\noindent with exact rows and commuting squares.  Here we identify $X_*(\bT_\bullet)$ with $T_\bullet/(T_\bullet\cap I_\bullet)$ by sending the cocharacter $\xi$ to the class of $\xi(\varpi)$.  

\begin{lem}\label{sequalsaff}
The diagram \eqref{diag} above induces an isomorphism $W_{\aff,\s}\cong W_\aff$, and an injection $\Omega_\s\longhookrightarrow \Omega$.  
\end{lem}

\begin{proof}
Let $\imath$ denote the injection $W_\s\longhookrightarrow W$, and note firstly that $\imath$ induces a bijection between the coroots of $\bT_\s$ and the coroots of $\bT$.  Let $\bbZ(\Phi_\bullet^\vee)\subset X_*(\bT_\bullet)$ denote the $\bbZ$-module generated by the coroots of $\bT_\bullet$.  %Given a root $\alpha\in\Phi$, let $\varphi_\alpha:\textbf{SL}_2\longrightarrow \bG$ denote the associated morphism.  Then $\alpha^\vee(t) = \varphi_\alpha(\textnormal{diag}(t,t^{-1}))$, and since $\textbf{SL}_2$ is simple, the image of $\varphi_\alpha$, and hence $\alpha^\vee$, must land in $\bG_\s$.  
Since the chosen vertex $x_0$ is hyperspecial, we have $\textnormal{stab}_{W_\bullet}(x_0)\cong W_{0,\bullet}$, which induces splittings of the short exact sequences above.  The discussion contained in Section 1.5 of \cite{Lu89} implies that the map
$$W_{\aff,\s} = \bbZ(\Phi_\s^\vee)\rtimes W_{0,\s}\stackrel{\imath}{\longrightarrow}\bbZ(\Phi^\vee)\rtimes W_0 = W_{\aff}$$
is an isomorphism.  One easily checks that the injection $\imath$ is compatible with the length function (using formula 1.4(a) in \emph{loc. cit.}, for example), which shows that the image of $\Omega_\s$ lies in $\Omega$.  
\end{proof}

\begin{cor}
The group $W_\s$ is normal in $W$, and $W_\s\backslash W$ admits coset representatives of length $0$.  
\end{cor}

We now consider the following diagram

\centerline{
\xymatrix{
1 \ar[r] & T_\s(k) \ar[r]\ar@{^{(}->}[d] & W_\s(1) \ar[r]\ar@{^{(}->}[d] & W_\s \ar[r]\ar@{^{(}->}[d] & 1\\
1 \ar[r] & T(k) \ar[r] & W(1) \ar[r] & W \ar[r] & 1
}
}
\noindent with exact rows and commuting squares.  A quick diagram chase verifies the following.

\begin{lem}\label{length0}
The group $W_\s(1)$ is normal in $W(1)$, and $W_\s(1)\backslash W(1)$ admits coset representatives of length $0$.  
\end{lem}

\begin{lem}\label{WGcoset}
Let $\cZ$ denote a closed subgroup of the connected center $Z$ of $G$ which satisfies $\cZ\cap I(1)G_\s = \{1\}$ (this implies that $\cZ\cap W_\s(1) = \{1\}$, the intersection taken inside $W(1)$).  Then the map 
\begin{eqnarray*}
W_\s(1)\backslash W(1)/\cZ & \longrightarrow & \cZ I(1)\backslash G/G_\s\\
 W_\s(1)w\cZ & \longmapsto & \cZ I(1)n_wG_\s
\end{eqnarray*}
induces a bijection of sets.  Indeed, it is an isomorphism of groups.
\end{lem}

\begin{proof}
It suffices to prove the claim with $\cZ = \{1\}$, since the assumption $\cZ\cap I(1)G_\s = \{1\}$ guarantees that $W_\s(1)\backslash W(1)$ fibers over $W_\s(1)\backslash W(1)/\cZ$ (resp. $I(1)\backslash G/G_\s$ fibers over $\cZ I(1)\backslash G/G_\s$) with fiber $\cZ$.  

Let $B\subset G$ denote the Borel subgroup defined by $\Phi^+$ containing $T$, with unipotent radical $U$, and let $U^-$ denote the opposite unipotent subgroup.  The Iwahori decomposition (\cite{Ti79}, Section 3.1.1) then implies
$$I(1) = (I(1)\cap T)(I(1)\cap U)(I(1)\cap U^-).$$
Since the group $G_\s$ is normal in $G$, we get $I(1)\backslash G/G_\s = I(1)G_\s\backslash G = (I(1)\cap T)G_\s\backslash G$, and a straightforward check shows that the group $(I(1)\cap T)G_\s$ is normal in $G$.  %For example, it suffices to show that conjugation by $I$ and $N$ preserves $(I(1)\cap T)G_\s$.  

One easily checks that the map $W_\s(1)\backslash W(1)\longrightarrow (I(1)\cap T)G_\s\backslash G$ is well-defined, and in fact defines a group homomorphism.  Since $N$ normalizes $I(1)\cap T$, the Bruhat and Iwahori decompositions imply
$$G = \bigsqcup_{w\in W(1)}I(1)n_w(I(1)\cap U)(I(1)\cap U^-),$$
which shows that the map is surjective.  Finally, assume that for $w\in W(1)$, we have $n_w = tg'$, $t\in I(1)\cap T, g'\in G_\s$.  This implies $t^{-1}n_w = g'\in G_\s\cap N = N_\s$, and therefore (by projecting to $W(1)$) we get $w\in W_\s(1)$.  This shows injectivity.  
\end{proof}

\section{Pro-$p$-Iwahori-Hecke Algebras}\label{algs}

\subsection{Structure of the algebras}

Let $\bullet$ denote the empty symbol or $\s$, and let $R$ denote an arbitrary commutative unital ring.  Given a smooth $R$-representation $\sigma$ of an open subgroup $J_\bullet$ of $G_\bullet$, we let $\textnormal{c-ind}_{J_\bullet}^{G_\bullet}(\sigma)$ denote the compact induction of $\sigma$ from $J_\bullet$ to $G_\bullet$.  That is, $\textnormal{c-ind}_{J_\bullet}^{G_\bullet}(\sigma)$ is the space of all functions $f:G_\bullet\longrightarrow\sigma$ for which $f(jg) = \sigma(j).f(g)$ for all $j\in J_\bullet, g\in G_\bullet$, such that the support of $f$ in $J_\bullet\backslash G_\bullet$ is compact.  The group $G_\bullet$ acts by right translation: if $f\in \textnormal{c-ind}_{J_\bullet}^{G_\bullet}(\sigma)$ and $g,g'\in G_\bullet$, we have $(g.f)(g') = f(g'g)$.

We consider the $R$-representation of $G_\bullet$ afforded by the \emph{pro-$p$ universal module}, defined by
$$\cC_{\bullet} := \textnormal{c-ind}_{I_\bullet(1)}^{G_\bullet}(1),$$
where $1$ denotes the free $R$-module of rank 1, with $I_\bullet(1)$ acting trivially.  For any element $g\in G_\bullet$, we denote by $\mathbf{1}_{I_\bullet(1)g}\in \cC_\bullet$ the characteristic function of the coset $I_\bullet(1)g$.

We define the \emph{pro-$p$-Iwahori-Hecke algebra} as
$$\cH_\bullet := \textnormal{End}_{G_\bullet}(\cC_\bullet)$$
with product given by composition.  The space $\cC_\bullet$ then becomes a left module over $\cH_\bullet$.  By Frobenius Reciprocity we have 
$$\cH_\bullet = \textnormal{End}_{G_\bullet}(\cC_\bullet)\cong \textnormal{Hom}_{I_\bullet(1)}(1,\cC_\bullet|_{I_\bullet(1)}) \cong \cC_\bullet^{I_\bullet(1)};$$
we therefore identify $\cH_\bullet$ with the $R$-module $R[I_\bullet(1)\backslash G_\bullet /I_\bullet(1)]$, with the product given by convolution.  For an element $g\in G_\bullet$, we let $\T_g^\bullet$ denote the characteristic function of the coset $I_\bullet(1)gI_\bullet(1)$.  By abuse of notation, we shall often speak of elements $\T_w^\bullet$, where $w\in W_\bullet(1)$; by the Bruhat decomposition (equation \eqref{bruhat}), this is independent of the choice of lift of $w$ to $N_\bullet$.  

We shall need one more algebra.  Let $\cZ$ be a closed subgroup of the connected center $Z$ of $G$ which satisfies $\cZ\cap I(1)G_\s = \{1\}$ (we allow the case $\cZ = \{1\}$).  Note that, when viewed as a subgroup of $W$, the elements of $\cZ$ have length 0.  We define
$$\underline{\cC} := \textnormal{c-ind}_{\cZ I(1)}^{G}(1),$$
where $1$ denotes the free $R$-module of rank 1, with $\cZ I(1)$ acting trivially.  We set
$$\underline{\cH} := \textnormal{End}_G(\underline{\cC});$$
by Frobenius Reciprocity, the algebra $\underline{\cH}$ identifies with $\underline{\cC}^{\cZ I(1)} = \underline{\cC}^{I(1)}$, with the product given by convolution of functions.  When $\cZ = \{1\}$, we have $\underline{\cC} = \cC$ and $\underline{\cH} = \cH$ as special cases.

Recall that we have an identification of $W_{\aff,\s}$ with $W_\aff$; we therefore identify the sets $S_\s$ and $S$.  For $s\in S$, there is an associated affine root $\alpha_\aff$ which is positive on $C$; we let $\alpha_s^\vee:\bG_m\longrightarrow\bT_\s\subset\bT$ denote the coroot associated to the ``root part'' of $\alpha_\aff$.  We set
$$\tau_s^\bullet := \sum_{a\in k^\times}\T_{\alpha_s^\vee(a)}^\bullet.$$

The structures of $\cH_\bullet$ and $\underline{\cH}$ are summarized in the following theorem.

\begin{thm}[\cite{Vig05}, Theorem 1 and \cite{OS11}, Section 5.1.3]\label{strthm}
 Let $\bullet$ denote either the empty symbol or $\s$.  
 \begin{enumerate}
  \item As an $R$-module, $\cH_\bullet$ is free with basis $\{\T_w^\bullet\}_{w\in W_\bullet(1)}$.
  \item (Braid relations)  We have 
  $$\T_w^\bullet\T_{w'}^\bullet = \T_{ww'}^\bullet$$
  for any $w,w'\in W_\bullet(1)$ satisfying $\ell(ww') = \ell(w) + \ell(w')$.  
  \item (Quadratic relations) For $s\in S$, we have
  $$(\T_{n_s}^\bullet)^2 = q\T_{n_s^2}^\bullet + \T_{n_s}^\bullet\tau_s^\bullet.$$
  \item Let $\Omega_\bullet(1)$ denote the preimage of $\Omega_\bullet$ under the natural projection $W_\bullet(1)\longrightarrow W_\bullet$.  Then the algebra $\cH_\bullet$ is generated by $\T_{n_s}^\bullet$ and $\T_\omega^\bullet$, where $s\in S$ and $\omega\in \Omega_\bullet(1)$. 
  \item We have an isomorphism of algebras
  $$\underline{\cH} \cong \cH/(\T_z - 1)_{z\in\cZ},$$
  which sends the characteristic function of $\cZ I(1)wI(1)$ to (the image of) the characteristic function of $I(1)wI(1)$.  
 \end{enumerate}
\end{thm}

For future applications, we will also need the affine subalgebra of $\cH_\bullet$.  

\begin{defn}\label{defaff}
  Let $W_{\aff,\bullet}(1)$ denote the preimage of $W_{\aff,\bullet}$ under the natural projection $W_\bullet(1)\longrightarrow W_\bullet$.  We denote by $\cH_{\aff,\bullet}$ the $R$-submodule of $\cH_\bullet$ generated by $\T_w$ for $w\in W_{\aff,\bullet}(1)$.  By Corollary 3 of \cite{Vig05}, $\cH_{\aff,\bullet}$ is a subalgebra of $\cH_\bullet$, called the \emph{affine pro-$p$-Iwahori-Hecke algebra}.  
  \end{defn}

\begin{remark}
By Theorem \ref{strthm}, we see that $\cH_{\aff,\bullet}$ is generated by the elements $\T_{n_s}$ and $\T_t$ for $s\in S$ and $t\in T_\bullet(k)$.  
\end{remark}

We now relate the various Hecke algebras.  We are mainly interested in how $\cH_\s$ is related to $\cH$; however, we give the proofs for the algebras $\cH_\s$ and $\underline{\cH}$, noting that $\underline{\cH} = \cH$ in the case $\cZ = \{1\}$.  

\begin{propn}\label{embedding}
Let $\overline{\T_g}$ denote the image of $\T_g$ in $\cH/(\T_z - 1)_{z\in\cZ}\cong \underline{\cH}$.  Then the linear map defined by
 \begin{center}
 \begin{tabular}{rrcl}
   $\mathfrak{f}:$& $\cH_\s$ & $\longrightarrow$ & $\underline{\cH}$\\
    & $\T_g^\s$ & $\longmapsto$ & $\overline{\T_g}$
 \end{tabular}
 \end{center}
 where $g\in G_\s$, is an injective algebra homomorphism.  
\end{propn}

\begin{proof}
The $G_\s$-linear map defined by
\begin{center}
 \begin{tabular}{rrcl}
   $\mathfrak{f}$: & $\cC_\s$ & $\longrightarrow$ & $\underline{\cC}|_{G_\s}$\\
    & $\mathbf{1}_{I_\s(1)}$ & $\longmapsto$ & $\mathbf{1}_{\cZ I(1)}$
 \end{tabular}
 \end{center}
is easily seen to be injective.  Taking $I_\s(1)$-invariants gives the injection
$$\cH_\s\cong \cC_\s^{I_\s(1)}\stackrel{\mathfrak{f}}{\longhookrightarrow} \underline{\cC}^{I_\s(1)}  = \underline{\cC}^{\cZ I(1)} \cong \underline{\cH},$$
which sends $\T_g^\s$ to $\overline{\T_g}$ for $g\in G_\s$.

It remains to check compatibility of $\mathfrak{f}$ with the algebra structures.  The injection $W_\s(1)\longhookrightarrow W(1)$ is compatible with the length function $\ell$.  Hence, if $w,w'\in W_\s(1)$ satisfy $\ell(ww') = \ell(w) + \ell(w')$, we get
$$\mathfrak{f}(\T_w^\s)\mathfrak{f}(\T_{w'}^\s) = \overline{\T_w\T_{w'}} = \overline{\T_{ww'}} = \mathfrak{f}(\T_{ww'}^\s) = \mathfrak{f}(\T_w^\s\T_{w'}^\s).$$
In addition, 
$$\mathfrak{f}(\tau_s^\s) = \overline{\tau_s}$$
for $s\in S$.  Since $\ell(t) = 0$ for $t\in T_\s(k)$, we obtain
\begin{eqnarray*}
(\mathfrak{f}(\T_{n_s}^\s))^2 & = & \overline{\T_{n_s}}^2\\
 & = & q\overline{\T_{n_s^2}} + \overline{\T_{n_s}\tau_s}\\
 & = & q\mathfrak{f}(\T_{n_s^2}^\s) + \mathfrak{f}(\T_{n_s}^\s)\mathfrak{f}(\tau_s^\s)\\
 & = & \mathfrak{f}(q\T_{n_s^2}^\s + \T_{n_s}^\s\tau_s^\s)\\
 & = & \mathfrak{f}((\T_{n_s}^\s)^2).
\end{eqnarray*}
\end{proof}

Using the proposition above, we shall henceforth identify $\cH_\s$ with its images in $\cH$ and $\underline{\cH}$.  We also fix a set of length 0 representatives $\cR$ for $W_\s(1)\backslash W(1)/\cZ$, which contains $1$ (cf. Lemmas \ref{length0} and \ref{WGcoset}).

\begin{lem}
Let $\omega\in \cR$.  Then the subspace $\cH_\s\overline{\T_\omega}$ of $\underline{\cH}$ is stable by left and right multiplication by $\cH_\s$.  
\end{lem}

\begin{proof}
We prove the claim for right multiplication, the case of left multiplication being obvious.  Let $w\in W_\s(1)\subset W(1)$.  One easily checks that $\ell(\omega w\omega^{-1}) = \ell(w)$, and we obtain
$$\overline{\T_{\omega}\T_w} = \overline{\T_{\omega w}} = \overline{\T_{\omega w\omega^{-1}}\T_{\omega}}\in \cH_\s\overline{\T_{\omega}},$$ 
which suffices to prove the claim.  
\end{proof}

\begin{propn}\label{freeness}
As a left (resp. right) $\cH_\s$-module, $\underline{\cH}$ is free with basis 
 $\{\ol{\T_{\omega}}\}_{\omega\in\cR}$.  
\end{propn}

\begin{proof}
This follows immediately from Lemma \ref{length0} and the braid relations of Theorem \ref{strthm}.   
 \end{proof}

\begin{cor}\label{dirsumdecomp}
The algebra $\cH_\s$ is a direct summand of $\underline{\cH}$ as a left (resp. right) $\cH_\s$ module.  
\end{cor}

We now use the above results to relate the modules $\cC_\s$, $\cC$, and $\underline{\cC}$.

\begin{lem}\label{univmodres}
There exists an isomorphism
$$\underline{\cH}\otimes_{\cH_\s}\cC_\s \cong \underline{\cC}|_{G_\s}$$
which is both $G_\s$-equivariant and $\underline{\cH}$-equivariant.
\end{lem}

\begin{proof}
Recall the map $\mathfrak{f}$ defined in the proof of Proposition \ref{embedding}:
\begin{center}
 \begin{tabular}{rrcl}
   $\mathfrak{f}$: & $\cC_{\s}$ & $\longrightarrow$ & $\underline{\cC}|_{G_\s}$\\
    & $\mathbf{1}_{I_\s(1)}$ & $\longmapsto$ & $\mathbf{1}_{\cZ I(1)}$
 \end{tabular}
 \end{center}
It is obviously $G_\s$- and $\cH_\s$-equivariant.  By Frobenius Reciprocity, we obtain a map $\widetilde{\mathfrak{f}}$, defined by
\begin{center}
 \begin{tabular}{rrcl}
   $\widetilde{\mathfrak{f}}$: & $\underline{\cH}\otimes_{\cH_\s}\cC_{\s}$ & $\longrightarrow$ & $\underline{\cC}|_{G_\s}$\\
    & $\overline{\T_{w}}\otimes g.\mathbf{1}_{I_\s(1)}$ & $\longmapsto$ & $\overline{\T_w}(\mathfrak{f}(g.\mathbf{1}_{I_\s(1)})) = g.\overline{\T_w}(\mathbf{1}_{\cZ I(1)})$
 \end{tabular}
 \end{center}
for $g\in G_\s$ and $w\in W(1)$.  The map $\widetilde{\mathfrak{f}}$ is $G_\s$- and $\underline{\cH}$-equivariant.  It remains to show that it is an isomorphism.

Using the Mackey decomposition and Lemma \ref{WGcoset}, we obtain
$$\underline{\cC}|_{G_\s} \cong \bigoplus_{\omega\in \cR} \textnormal{c-ind}_{\cZ I(1)}^{\cZ I(1)n_\omega G_\s}(1),$$
where $\textnormal{c-ind}_{\cZ I(1)}^{\cZ I(1)n_\omega G_\s}(1)$ denotes the subspace of $\underline{\cC}$ with support contained in $\cZ I(1)n_\omega G_\s$.  Analogously, $\{\overline{\T_{\omega}}\}_{\omega\in\cR}$ is a basis for $\underline{\cH}$ over $\cH_\s$, and we obtain
$$\underline{\cH}\otimes_{\cH_\s}\cC_\s \cong \bigoplus_{\omega\in\cR}\overline{\T_{\omega}}\otimes_{\cH_\s}\cC_\s.$$
It is clear that $\widetilde{\mathfrak{f}}$ defines an isomorphism between $\overline{\T_{\omega}}\otimes_{\cH_\s}\cC_\s$ and $\textnormal{c-ind}_{\cZ I(1)}^{\cZ I(1)n_\omega G_\s}(1)$.  
\end{proof}

\begin{cor}\label{easyisom}
Let $\MM$ be a right $\underline{\cH}$-module.  We then have an isomorphism 
 $$\MM|_{\cH_\s}\otimes_{\cH_\s}\cC_\s\cong \MM\otimes_{\underline{\cH}} \underline{\cC}|_{G_\s}$$
as $G_\s$-representations.
\end{cor}

\begin{cor}\label{cflatiffcsflat}
The module $\cC$ is flat (resp. projective) over $\cH$ if and only if $\cC_\s$ is flat (resp. projective) over $\cH_\s$, if and only if $\underline{\cC}$ is flat (resp. projective) over $\underline{\cH}$.  
\end{cor}

\begin{proof}
Using the fact that $\cH$ and $\underline{\cH}$ are free over $\cH_\s$, this follows from a straightforward descent argument.  
%
%Base extension preserves both flatness and projectivity.  For the reverse direction, let $\mm$ be a left $\cH_\s$-module.  As a left $\cH_\s$-module, we have $\cH\otimes_{\cH_\s}\mm \cong \bigoplus_{\omega\in\cR}\T_\omega\otimes\mm \cong \bigoplus_{\omega\in\cR}\mm^\omega$, where $\mm^\omega$ is the left $\cH_\s$-module $\mm$ with action twisted by the automorphism $\T_w\longmapsto \T_{\omega^{-1}w\omega}$.  
%
%Assume now that $\cH\otimes_{\cH_\s}\mm$ is flat.  Since $\cH$ is flat over $\cH_\s$, we have that $\cH\otimes_{\cH_\s}\mm$ is flat as a left $\cH_\s$-module.  Since a direct sum is flat if and only if each constituent is flat, we conclude that $\mm$ is flat as a left $\cH_\s$-module.  
%
%Assume now $\cH\otimes_{\cH_\s}\mm$ is projective, so that there exists a left $\cH$-module $\mathfrak{N}$ such that $(\cH\otimes_{\cH_\s}\mm)\oplus \mathfrak{N} \cong \cH^{\oplus J}$ for some index set $J$.  Restricting to $\cH_\s$, we get $(\bigoplus_{\omega\in\cR}\mm^\omega)\oplus\mathfrak{N}|_{\cH_\s} \cong (\cH_\s^{\oplus \cR})^{\oplus J}$; since $\mm$ is a direct summand of a free $\cH_\s$-module, projectivity follows.  
\end{proof}

\begin{cor}\label{flatcor}
Suppose that $R = \fpb$.  Let $\bG = \mathbf{GL}_n$, so that $G = \textnormal{GL}_n(F)$ and $G_\s = \textnormal{SL}_n(F)$, and let $\cZ$ denote the subgroup of $G$ consisting of central elements whose entries are powers of $\varpi$.   
\begin{enumerate}
 \item Let $n = 2$.  Then $\cC_\s$ is flat if and only if $q = p$, in which case it is projective and faithfully flat.  
 \item Let $n = 3$ and assume $q = p > 2$.  Then $\cC_\s$ is not flat.
\end{enumerate}
\end{cor}

\begin{proof}
 Using the previous corollary, part (1) follows from \cite{Oll07}, Th\'{e}or\`{e}me 1, while part (2) follows from \cite{OS11}, Proposition 7.9.  
% 
% Faithful flatness requires some comment.  By \cite{Oll07}, when $q = p$ we have that $\underline{\cH}$ is a direct summand of $\underline{\cC}$, say $\underline{\cC}\cong \underline{\cH}\oplus \underline{\MM}$.  Since $\underline{\cC}$ is projective, it is a direct summand of a free module, and the same being true of $\underline{\MM}$ implies that $\underline{\MM}$ is projective.  Now, we have an exact sequence
% $$0\longrightarrow \cH_\s\longrightarrow \cC_\s\longrightarrow \cC_\s/\cH_\s\longrightarrow 0,$$
% which, upon tensoring with $\underline{\cH}$, gives
% $$0\longrightarrow \underline{\cH}\longrightarrow \underline{\cC}\longrightarrow \underline{\cH}\otimes_{\cH_\s}(\cC_\s/\cH_\s)\longrightarrow 0.$$
% Therefore $\underline{\MM}\cong \underline{\cH}\otimes_{\cH_\s}(\cC_\s/\cH_\s)$ is projective, and using the descent argument above, we get that $\cC_\s/\cH_\s$ is projective.  Therefore, the original exact sequence splits, and $\cH_\s$ is a direct summand of $\cC_\s$.  Proceeding as in \cite{Oll07}, we get faithful flatness of $\cC_\s$.  
\end{proof}

\subsection{Iwahori--Hecke Algebras}
We also have a variant of the above setup.  Consider the universal module $\cC'_\bullet := \textnormal{c-ind}_{I_\bullet}^{G_\bullet}(1)$, and its endomorphism algebra $\cH'_\bullet := \textnormal{End}_{G_\bullet}(\cC'_\bullet)$, called the \emph{Iwahori-Hecke algebra}.  Its structure is well known (cf. \cite{Vig05}), and the proofs above apply mutatis mutandis to the Iwahori case.  We remark that some of the results relating $\cH_\s'$ and $\cH'$ have also been obtained by Abdellatif in \cite{Ab11} (for $\bG = \mathbf{GL}_n$).  Translating Corollary \ref{flatcor} to the present situation, we obtain the following.

\begin{cor}
Suppose $R = \fpb$, and let $\bG = \mathbf{GL}_2$, so that $G = \textnormal{GL}_2(F)$ and $G_\s = \textnormal{SL}_2(F)$.  Then $\cC'_\s$ is projective and faithfully flat.  
\end{cor}

\begin{proof}
We once again use \cite{Oll07}, Th\'{e}or\`{e}mes 1 and 2.
\end{proof}

\section{Equivalence of Categories between $G_\s$-representations and $\cH_\s$-modules}\label{equivs}

We assume from this point onwards that $R = \fpb$ and $\bG = \mathbf{GL}_n$ with $n\geq 2$, so that $\bG_\s = \mathbf{SL}_n, G = \textnormal{GL}_n(F)$, and $G_\s = \textnormal{SL}_n(F)$.  We take $\bT$ to be the diagonal maximal torus, $I$ the subgroup of $G$ with entries in $\oo$ which are upper-triangular modulo $\varpi$, and $I(1)$ those elements in $I$ which are unipotent modulo $\varpi$.  We take $\cZ$ to be the central subgroup consisting of diagonal matrices whose entries are a power of $\varpi$.

Let $\bullet$ denote either the empty symbol or $\s$.  Let $\pi$ be a smooth $\fpb$-representation of the group $G_\bullet$; Frobenius Reciprocity gives 
$$\pi^{I_\bullet(1)}\cong \textrm{Hom}_{I_\bullet(1)}(1,\pi|_{I_\bullet(1)})\cong \textnormal{Hom}_{G_\bullet}(\cC_{\bullet},\pi).$$ 
The algebra $\cH_\bullet$ has a natural right action on $\textrm{Hom}_{G_\bullet}(\cC_\bullet,\pi)$ by pre-composition, which induces a right action on $\pi^{I_\bullet(1)}$.  In this way, we obtain the functor of $I_\bullet(1)$-invariants
\begin{center}
\begin{tabular}{rrcl}
$\mathcal{I}_\bullet:$ & $\mathfrak{Rep}_{\fpb}(G_\bullet)$ & $\longrightarrow$ & $\mathfrak{Mod}-\cH_\bullet$\\
 & $\pi$ & $\longmapsto$ & $\pi^{I_\bullet(1)}$
\end{tabular}
\end{center}
from the category of smooth $\fpb$-representations of $G_\bullet$ to the category of right $\cH_\bullet$-modules.  

Let $\mathfrak{Rep}_{\fpb}^{I_\bullet(1)}(G_\bullet)$ denote the full subcategory of $\mathfrak{Rep}_{\fpb}(G_\bullet)$ of objects generated by their space of $I_\bullet(1)$-invariants.  We continue to denote by $\mathcal{I}_\bullet$ the functor above restricted to the subcategory $\mathfrak{Rep}_{\fpb}^{I_\bullet(1)}(G_\bullet)$.  Lemma 3 Part (1) of \cite{BL94} implies that this functor is faithful.

Given a right $\cH_\bullet$-module $\mm$, we may consider the $G_\bullet$-representation $\mm\otimes_{\cH_\bullet}\cC_\bullet$, with the action of $G_\bullet$ given on the right tensor factor.  We thus obtain a functor
\begin{center}
\begin{tabular}{rrcl}
$\mathcal{T}_\bullet:$ &  $\mathfrak{Mod}-\cH_\bullet$ & $\longrightarrow$ & $\mathfrak{Rep}_{\fpb}^{I_\bullet(1)}(G_\bullet)$\\
 & $\mm$ & $\longmapsto$ & $\mm\otimes_{\cH_\bullet}\cC_\bullet$.  
\end{tabular}
\end{center}
The functors $\mathcal{I}_\bullet$ and $\mathcal{T}_\bullet$ are adjoint to each other:
$$\textnormal{Hom}_{\mathfrak{Rep}_{\fpb}^{I_\bullet(1)}(G_\bullet)}(\mathcal{T}_\bullet(\mm),\pi) \cong \textnormal{Hom}_{\mathfrak{Mod}-\cH_\bullet}(\mm,\mathcal{I}_\bullet(\pi)).$$

We furthermore let $\mathfrak{Rep}_{\fpb}^{I(1)}(G)_{\varpi = 1}$ denote the full subcategory of $\mathfrak{Rep}_{\fpb}^{I(1)}(G)$ of representations on which the central element 
$$\begin{pmatrix}\varpi & & \\ & \ddots & \\ & & \varpi\end{pmatrix}$$
acts trivially.  Taking $I(1)$-invariants of a representation in $\mathfrak{Rep}_{\fpb}^{I(1)}(G)_{\varpi = 1}$ yields a right $\cH$-module on which the element $\T_{\textnormal{diag}(\varpi,\ldots,\varpi)}$ acts trivially.  Therefore, we obtain a functor 
\begin{center}
\begin{tabular}{rrcl}
$\underline{\mathcal{I}}:$ & $\mathfrak{Rep}_{\fpb}^{I(1)}(G)_{\varpi = 1}$ & $\longrightarrow$ & $\mathfrak{Mod}-\underline{\cH}$\\
 & $\pi$ & $\longmapsto$ & $\pi^{I(1)}$.
\end{tabular}
\end{center}
The adjoint functor $\underline{\mathcal{T}}$ is given by
\begin{center}
\begin{tabular}{rrcl}
$\underline{\mathcal{T}}:$ &  $\mathfrak{Mod}-\underline{\cH}$ & $\longrightarrow$ & $\mathfrak{Rep}_{\fpb}^{I(1)}(G)_{\varpi = 1}$\\
 & $\mm$ & $\longmapsto$ & $\mm\otimes_{\underline{\cH}}\underline{\cC}$.  
\end{tabular}
\end{center}

It is natural to ask whether the functors defined above induce equivalences of categories.  This shall be the main goal of this section.

\begin{lem}\label{I1Is1invts}
 Assume that $(n,q) = 1$ and let $\pi\in \mathfrak{Rep}_{\fpb}^{I(1)}(G)$.  We then have an equality of vector spaces $\pi^{I_\s(1)} = \pi^{I(1)}$.  
\end{lem}

\begin{proof}
Consider the short exact sequence 
$$1\longrightarrow I_\s(1)\longrightarrow I(1)\stackrel{\det}{\longrightarrow} 1 + \pp \longrightarrow 1.$$
Since $n$ and $p$ are coprime, the map $x\longmapsto x^n$ is an automorphism of $1 + \pp$.  Hence, the map
$$x\longmapsto\begin{pmatrix}x^{1/n} &  & \\  & \ddots & \\ & & x^{1/n} \end{pmatrix}$$
gives an isomorphism $1 + \pp \cong I(1)\cap Z$, and a section to the surjection $I(1)\stackrel{\det}{\longrightarrow} 1 + \pp$.  Thus, we have a decomposition
$$I(1)\cong I_\s(1)\times (I(1)\cap Z).$$

Now, since $I(1)\cap Z$ is contained in $I(1)$, it acts trivially on $\pi^{I(1)}$.  As $\pi$ is generated by its space of $I(1)$-invariants, $I(1)\cap Z$ acts trivially on the whole of $\pi$.  Hence, by the above decomposition, we get $\pi^{I_\s(1)} \subset \pi^{I(1)}$.   The opposite inclusion is obvious.  
\end{proof}

\begin{remark}
 By considering the action of $\cH$ on $\pi^{I(1)}$ and restricting to $\cH_\s$, the above lemma actually yields an isomorphism of $\cH_\s$-modules 
\begin{equation}\label{HandHsmods}
\pi^{I_\s(1)}\cong \pi^{I(1)}|_{\cH_\s}.
\end{equation}
\end{remark}

\begin{thm}\label{equiv}
 Assume that $(n,q) = 1$.  The functors $\mathcal{I}$ and $\mathcal{T}$ induce an equivalence of categories between $\mathfrak{Rep}_{\fpb}^{I(1)}(G)$ and $\mathfrak{Mod}-\cH$ if and only if $\mathcal{I}_\s$ and $\mathcal{T}_\s$ induce and equivalence of categories between $\mathfrak{Rep}_{\fpb}^{I_\s(1)}(G_\s)$ and $\mathfrak{Mod}-\cH_\s$, if and only if $\underline{\mathcal{I}}$ and $\underline{\mathcal{T}}$ induce an equivalence of categories between $\mathfrak{Rep}_{\fpb}^{I(1)}(G)_{\varpi = 1}$ and $\mathfrak{Mod}-\underline{\cH}$.  
\end{thm}

\begin{proof}
The proofs of both equivalences are similar.  We prove the first statement.  

 ($\Longrightarrow$) Assume first that $\mathcal{I}$ and $\mathcal{T}$ induce an equivalence of categories, and let $\mm$ be a right $\cH_\s$-module.  We will show that the natural map
\begin{center}
 \begin{tabular}{rcl}
  $\mm$ & $\longrightarrow$ & $\mathcal{I}_\s\circ\mathcal{T}_\s(\mm) = (\mm\otimes_{\cH_\s}\cC_\s)^{I_\s(1)}$\\
 $m$ & $\longmapsto$ & $m\otimes\mathbf{1}_{I_\s(1)}$
 \end{tabular}
\end{center}
is an isomorphism.  

Let $\mmm$ denote the induced $\cH$-module $\mm\otimes_{\cH_\s}\cH$.  Since $\mathcal{I}$ and $\mathcal{T}$ induce an equivalence, the homomorphism
\begin{center}
 \begin{tabular}{rcl}
  $\mmm$ & $\longrightarrow$ & $\mathcal{I}\circ\mathcal{T}(\mmm) = (\mmm\otimes_{\cH}\cC)^{I(1)}$\\
 $m$ & $\longmapsto$ & $m\otimes\mathbf{1}_{I(1)}$
 \end{tabular}
\end{center}
is bijective.  On restricting to $\cH_\s$, we get 
\begin{eqnarray*}
 \mmm|_{\cH_\s} & \cong & (\mmm\otimes_\cH\cC)^{I(1)}|_{\cH_\s}\\
 & \stackrel{\textnormal{eq.}~\eqref{HandHsmods}}{\cong} & (\mmm\otimes_{\cH}\cC|_{G_\s})^{I_\s(1)}\\
 & \stackrel{\textnormal{Cor.}~\ref{easyisom}}{\cong} & (\mmm|_{\cH_\s}\otimes_{\cH_\s}\cC_\s)^{I_\s(1)}.
\end{eqnarray*}
Since $\cH_\s$ is a direct summand of $\cH$, we have $\mmm|_{\cH_\s}\cong \mm\oplus\mm',$ where $\mm'$ is the right $\cH_\s$-module $\bigoplus_{\omega\in\cR, \omega\neq 1}\mm\otimes_{\cH_\s} \T_\omega$.  Therefore, the above isomorphisms give
$$\mm\oplus\mm' \cong (\mm\otimes_{\cH_\s}\cC_\s)^{I_\s(1)} \oplus (\mm'\otimes_{\cH_\s}\cC_\s)^{I_\s(1)};$$
since the image of $\mm$ (resp. $\mm'$) under the natural map $\mmm\longrightarrow (\mmm\otimes_\cH\cC)^{I(1)}$ must lie in the space $(\mm\otimes_{\cH_\s}\cC_\s)^{I_\s(1)}$ (resp. $(\mm'\otimes_{\cH_\s}\cC_\s)^{I_\s(1)}$), we conclude that 
$$\mm\cong (\mm\otimes_{\cH_\s}\cC_\s)^{I_\s(1)} = \mathcal{I}_{\s}\circ\mathcal{T}_\s(\mm).$$

Now let $\pi\in\mathfrak{Rep}_{\fpb}^{I_\s(1)}(G_\s)$, and consider the natural map
\begin{center}
 \begin{tabular}{rcl}
$\mathcal{T}_\s\circ\mathcal{I}_\s(\pi) = \pi^{I_\s(1)}\otimes_{\cH_\s}\cC_\s$ & $\longrightarrow$ & $\pi$\\
$v\otimes \mathbf{1}_{I_\s(1)g^{-1}}$ & $\longmapsto$ & $g.v$,
 \end{tabular}
\end{center}
where $v\in \pi^{I_\s(1)}$ and $g\in G_\s$.  Since $\pi$ is generated by its space of $I_\s(1)$-invariant vectors, this map is surjective.  Letting $\pi'$ denote its kernel, we obtain a short exact sequence
$$0\longrightarrow \pi'\longrightarrow \pi^{I_\s(1)}\otimes_{\cH_\s}\cC_\s \longrightarrow \pi \longrightarrow 0.$$
Taking $I_\s(1)$-invariants of this short exact sequence yields 
$$0\longrightarrow (\pi')^{I_\s(1)}\longrightarrow (\pi^{I_\s(1)}\otimes_{\cH_\s}\cC_\s)^{I_\s(1)} \longrightarrow \pi^{I_\s(1)},$$
and the statement just proved shows that the third arrow is an isomorphism.  Hence $(\pi')^{I_\s(1)} = 0$, and faithfulness of the functor $\mathcal{I}_\s$ shows $\pi' = 0$.  We conclude that 
$$\mathcal{T}_\s\circ\mathcal{I}_\s(\pi) = \pi^{I_\s(1)}\otimes_{\cH_\s}\cC_\s \cong \pi.$$

($\Longleftarrow$) Assume now that $\mathcal{I}_\s$ and $\mathcal{T}_\s$ induce an equivalence of categories, and let $\MM\in \mathfrak{Mod}-\cH$.  Consider the natural map
\begin{center}
 \begin{tabular}{rcl}
  $\MM$ & $\longrightarrow$ & $\mathcal{I}\circ\mathcal{T}(\MM) = (\MM\otimes_{\cH}\cC)^{I(1)}$\\
 $m$ & $\longmapsto$ & $m\otimes\mathbf{1}_{I(1)}$.
 \end{tabular}
\end{center}
We claim that this map is an isomorphism.  Indeed, if we restrict this morphism to $\cH_\s$, we obtain
\begin{eqnarray*}
 \MM|_{\cH_\s} & \longrightarrow  & (\MM\otimes_{\cH}\cC)^{I(1)}|_{\cH_\s}\\
 & & \stackrel{\textnormal{eq.}~\eqref{HandHsmods}}{\cong} (\MM\otimes_\cH\cC|_{G_\s})^{I_\s(1)}\\
 & & \stackrel{\textnormal{Cor.}~\ref{easyisom}}{\cong} (\MM|_{\cH_\s}\otimes_{\cH_\s}\cC_\s)^{I_\s(1)}.
\end{eqnarray*}
Since $\mathcal{I}_\s$ and $\mathcal{T}_\s$ induce an equivalence of categories, this map is an isomorphism, and we obtain
$$\MM \cong (\MM\otimes_{\cH}\cC)^{I(1)} = \mathcal{I}\circ\mathcal{T}(\MM).$$

Now let $\Pi\in\mathfrak{Rep}_{\fpb}^{I(1)}(G)$.  In order to show that $\mathcal{T}\circ\mathcal{I}$ is naturally isomorphic to $\textnormal{id}_{\mathfrak{Rep}_{\fpb}^{I(1)}(G)}$, we proceed exactly as in the proof of $\mathcal{T}_\s\circ\mathcal{I}_\s\simeq\textnormal{id}_{\mathfrak{Rep}_{\fpb}^{I_\s(1)}(G_\s)}$ above.  Hence, we conclude
$$\mathcal{T}\circ\mathcal{I}(\Pi) = \Pi^{I(1)}\otimes_{\cH}\cC \cong \Pi.$$
\end{proof}

\begin{cor}\label{equivcor}  
Let $\mathcal{I}:\mathfrak{Rep}_{\fpb}^{I(1)}(G)\longrightarrow\mathfrak{Mod}-\cH$ denote the functor of $I(1)$-invariants.
\begin{enumerate}
 \item The functor $\mathcal{I}$ induces an equivalence of categories for $n = 2$ and $F = \qp$ with $p > 2$.  
\end{enumerate}
Let $\mathcal{I}_\s:\mathfrak{Rep}_{\fpb}^{I_\s(1)}(G_\s)\longrightarrow\mathfrak{Mod}-\cH_\s$ denote the functor of $I_\s(1)$-invariants.
\begin{enumerate}\setcounter{enumi}{+1}
 \item The functor $\mathcal{I}_\s$ induces an equivalence of categories for $n = 2$ and $F = \qp$ with $p > 2$.  
 \item The functor $\mathcal{I}_\s$ \emph{does not} induce an equivalence of categories when $n = 2$ and $q > p > 2$.
 \item The functor $\mathcal{I}_\s$ \emph{does not} induce an equivalence of categories when $n = 2$ and $F = \mathbb{F}_p((T))$ with $p>2$.  
 \item The functor $\mathcal{I}_\s$ \emph{does not} induce an equivalence of categories when $n = 3$ and $q = p > 3$.
\end{enumerate}
\end{cor}

\begin{proof}
 Using Theorem \ref{equiv} above, parts (1) - (4) follow from Th\'eor\`eme 1.3 of \cite{Oll09}, and part (5) follows from Corollary \ref{flatcor}.  
\end{proof}

\section{$L$-Packets of $\cH_\s$-modules}\label{packets}

Once again, we let $\bullet$ denote either the empty symbol or $\s$ throughout.  We introduce some additional notation.  For $1\leq i \leq n - 1$, we define
$$n_i := \begin{pmatrix}1 & & & & & \\ & \ddots & & & & \\ & & \phantom{-}0 & 1 & & \\ & & -1 & 0 & & \\ & &  &  & \ddots & \\ & &  &  & & 1 \end{pmatrix}\begin{array}{c} \phantom{-} \\ \phantom{-} \\ i^{\textnormal{th}}~\textnormal{row} \\ (i + 1)^{\textnormal{th}}~\textnormal{row}\\ \phantom{-} \\ \phantom{-} \end{array}\quad\textnormal{and}\quad n_0 := \begin{pmatrix} & & & & -\varpi^{-1}\\ & 1 & & & \\ & & \ddots &  & \\ & &  &  1 & \\ \varpi & &  &  &  \end{pmatrix}.$$
If we choose for $x_0$ the hyperspecial vertex corresponding to the subgroup $\textnormal{GL}_n(\oo)$, then the set $S$ of affine reflections generating $W_\aff \cong W_\s$ is given by (the images of) the elements $n_i$, $0\leq i \leq n - 1$.  
Finally, we let
$$\omega := \begin{pmatrix}0 & 1 & & \\ & \ddots  & \ddots & \\  &  & \ddots & 1\\ \varpi &  & & 0 \end{pmatrix}.$$
The element $\omega$ satisfies 
$$\omega^{-1}n_i\omega = n_{i + 1}$$
(the index being considered modulo $n$), and is a generator for the group $\Omega$ of length $0$ elements of $W$.

\subsection{Supersingular Hecke Modules}  We recall the following definition of supersingular modules.

\begin{defn}[\cite{Vig05}, Definition 3]
 Let $\mm$ be a nonzero right $\cH_\bullet$-module.  Then $\mm$ is said to be \emph{supersingular} if the center of $\cH_\bullet$ acts by a character which is null.  
\end{defn}

For a more precise definition of a null character of the center of $\cH_\bullet$, we refer to Definition 2 (\emph{loc. cit.}).  In order to describe the supersingular modules more explicitly, we recall the classification of characters of $\cH_{\aff,\bullet}$.

Recall that to each $n_i\in S$, we have an associated coroot $\alpha_{n_i}^\vee:\bG_m\longrightarrow \bT_\s\subset \bT$.  Let $\lambda:T_\bullet(k)\longrightarrow \fpb^\times$ be a character, and set 
$$S_\lambda := \{n_i\in S:~\lambda\circ\alpha_{n_i}^\vee:k^\times\longrightarrow \fpb^\times ~\textnormal{is trivial}\}.$$

\begin{propn}[\emph{loc. cit.}, Proposition 2]
 The characters of $\cH_{\aff,\bullet}$ are parametrized by pairs $(\lambda,J)$, where $\lambda:T_\bullet(k)\longrightarrow \fpb^\times$ is a character and $J\subset S_\lambda$.  We denote the character associated to the pair $(\lambda,J)$ by $\chi_{\lambda,J}^\bullet$.  This character is defined by
\begin{center}
\begin{tabular}{ccccl}
 (1) & $\chi_{\lambda,J}^\bullet(\T_t)$ & $=$ & $\lambda(t)$ & for $t\in T_\bullet(k)$,\\
 (2) & $\chi_{\lambda,J}^\bullet(\T_{n_i})$ & $=$ & $0$ & if $n_i\not\in J$,\\
 (3) & $\chi_{\lambda,J}^\bullet(\T_{n_i})$ & $=$ & $-1$ & if $n_i\in J$.  
\end{tabular}
\end{center}
\end{propn}

Let $1$ denote the trivial character of $T_\bullet(k)$.  The proposition above shows, in particular, that the algebra $\cH_{\aff,\bullet}$ possesses two distinguished characters: the \emph{trivial character} $\chi_{1,\emptyset}^\bullet$, sending all $\T_{n_i}$ to 0, and the \emph{sign character} $\chi_{1,S}^\bullet$, sending all $\T_{n_i}$ to $-1$.

We have the following classification of supersingular modules.  

\begin{thm}[\cite{Oll12}, Theorem 5.14]\label{class}
Let $\mm$ be a simple right $\cH_\bullet$-module.  Then $\mm$ is supersingular if and only if it contains a character of $\cH_{\aff,\bullet}$, which is different from a twist of $\chi_{1,\emptyset}^\bullet$ or $\chi_{1,S}^\bullet$.  
\end{thm}

Since $\cH_{\aff,\s} = \cH_\s$, this theorem implies in particular that every supersingular module of $\cH_\s$ is a character, unequal to $\chi_{1,\emptyset}^\s$ or $\chi_{1,S}^\s$.  

\subsection{$L$-packets}
Consider now a simple supersingular module $\MM$ for $\cH$.  By Theorem \ref{class}, $\MM|_{\cH_\aff}$ contains a character.  Restricting to $\cH_\s$, we obtain a character $\chi$ of $\cH_\s$, which furthermore must be supersingular. By the construction of the character $\chi$, the underlying vector space is stable by elements of the form $\T_{t}$ for $t\in T(k)$, and we obtain 
$$\MM|_{\cH_\s} = \sum_{j = 0}^{n - 1}\chi\cdot\T_{\omega^j}$$
by simplicity of $\MM$.

The element $\omega$ acts on the set $S$ by conjugation, and we see that
\begin{equation}\label{omorbit}
\chi_{\lambda,J}^\bullet\cdot\T_\omega\cong \chi_{\lambda^\omega,\omega^{-1}J\omega}^\bullet,
\end{equation}
where $\lambda^\omega$ is defined by $\lambda^{\omega}(t) = \lambda(\omega t\omega^{-1})$ for $t\in T_\bullet(k)$.  This leads to the following definition.

\begin{defn}
 Let $\lambda:T_\s(k)\longrightarrow \fpb^\times$ be a character, and $J\subset S_\lambda \subset S$.  We define an action of $\omega^{\mathbb{Z}}$ on the characters of $\cH_{\s}$ by 
$$\omega.\chi_{\lambda,J}^\s := \chi_{\lambda^\omega,\omega^{-1} J\omega}^\s \cong \chi_{\lambda,J}^\s\cdot\T_\omega.$$
We define an \emph{L-packet of $\cH_\s$-modules} to be an orbit of $\omega^\mathbb{Z}$ acting on characters of $\cH_\s$.  We say an $L$-packet is \emph{supersingular} if it consists entirely of supersingular characters, or, equivalently, if it contains a supersingular character.  
\end{defn}

In particular, we see that the size of an $L$-packet must divide $n$.

\begin{defn}\label{regular}
 We say a supersingular character $\chi_{\lambda,J}^\s$ is \emph{regular} if there exists a simple supersingular $\cH$-module $\MM$ of dimension $n$ such that $\chi_{\lambda,J}^\s$ is a Jordan-H\"older factor of $\MM|_{\cH_\s}$.  We say an $L$-packet is \emph{regular} if every character contained in the packet is regular, or, equivalently, if it contains a regular character.  
\end{defn}

It is an easy exercise to see that if $\MM$ is a simple $n$-dimensional supersingular $\cH$-module, the restriction $\MM|_{\cH_\aff}$ is a direct sum of $n$ distinct characters.  This implies that $\chi_{\lambda,J}^\s$ is regular if and only if, for any character $\widetilde{\lambda}:T(k)\longrightarrow \fpb^\times$ satisfying $\widetilde{\lambda}|_{T_\s(k)} = \lambda$, the orbit of the character $\chi_{\widetilde{\lambda},J}$ of $\cH_\aff$ has size $n$ under the action of $\omega^\mathbb{Z}$ (where the action of $\omega^\mathbb{Z}$ on $\chi_{\widetilde{\lambda},J}$ is defined by equation \eqref{omorbit}).

In what follows, we let $(e_1, \ldots, e_j)$ denote the greatest common divisor of integers $e_1,\ldots, e_j\in \mathbb{Z}$, with $(e_1) = |e_1|$.  In addition, for any natural number $a\in\bbN$, we shall denote by 
$$[a] = \frac{q^a - 1}{q - 1}$$
the $q$-analog of $a$.

To proceed further, we need a combinatorial lemma.  

\begin{lem}\label{comb}
 Let $f:\mathbb{Z}_{>0}\longrightarrow \mathbb{C}$ be an arbitrary arithmetic function, let $\mu:\mathbb{Z}_{>0}\longrightarrow \{-1,0,1\}$ denote the M\"obius function, and  let $\sigma_0(m)$ denote the number of divisors of $m$.  We then have
$$f(m) - \sum_{j=1}^{\sigma_0(m) - 1}(-1)^{j + 1}\sum_{\sub{1\leq e_1 < \ldots < e_j < m}{e_i|m}}f((e_1, \ldots, e_j)) = \sum_{e|m}\mu\left(\frac{m}{e}\right)f(e).$$
\end{lem}

\begin{proof}
It suffices to show
$$\mu(m') = \sum_{j = 1}^{\sigma_0(m') - 1}(-1)^j|\{1\leq e_1' < \ldots < e_j' < m': e_i'|m',~(e_1', \ldots, e_j') = 1\}|;$$
this follows from (the example following) Proposition 4.29 in \cite{Aig97}.  %One takes $P$ to be the distributive lattice of divisors of $m'$, with partial order given by $x\preceq y$ if and only if $y$ divides $x$.  
\end{proof}

\begin{propn}\label{orbdivd}
 Let $d$ be a divisor of $n$, and let $g(d)$ denote the number of regular supersingular characters of $\cH_\s$ whose orbit under $\omega^\mathbb{Z}$ has size dividing $d$.  We then have $$g(d) = \sum_{e|n} \mu\left(\frac{n}{e}\right)[(d,e)]\left(\frac{e}{(d,e)}, q - 1\right).$$
\end{propn}

\begin{proof}
Let $\chi_{\lambda,J}^\s$ be a supersingular character whose orbit under $\omega^\mathbb{Z}$ has size dividing $d$.  This means that $\omega^{-d}J\omega^d = J$, that is, the set $J$ is stable under the map $n_i \longmapsto n_{i+d}$.  Hence, the subsets $J$ of $S$ satisfying $\omega^{-d}J\omega^d = J$ correspond bijectively to subsets $J'$ of $\{n_1, \ldots, n_d\}$ in the obvious way.  

Let $\lambda$ correspond to the equivalence class
$$((a_1, a_2, \ldots, a_n))\in (\mathbb{Z}/(q-1)\mathbb{Z})^n/\langle(1,1,\ldots, 1)\rangle;$$
the correspondence is defined by $\lambda(\textnormal{diag}(t_1,t_2,\ldots, t_n)) = \prod_{i = 1}^nt_i^{a_i}$, where $t_i\in k^\times$ and $\prod_{i = 1}^n t_i = 1$.  The condition $\lambda^{\omega^d} = \lambda$ implies that there exists $z\in\mathbb{Z}/(q-1)\mathbb{Z}$ such that 
$$a_{i + d} \equiv a_{i} + z~(\textnormal{mod}~ q - 1)$$
for every $0 < i \leq n$ (where we consider the indices modulo $n$).  Summing gives
$$\sum_{j = 0}^{n/d - 1} a_{i + jd} \equiv \frac{n}{d}z + \sum_{j = 0}^{n/d - 1} a_{i + jd}~(\textnormal{mod}~ q - 1)~ \Longleftrightarrow~ z \equiv 0~ \left(\textnormal{mod}\frac{q - 1}{(\frac{n}{d}, q - 1)}\right).$$
Hence, $\lambda$ is determined by $(a_1, \ldots, a_d)$ modulo the diagonal, and the element $z$.  
%The class corresponding to $\lambda$ takes the form
%$$\left(\left(a_1, \ldots, a_d, a_1 + z, \ldots, a_d + z, \ldots, a_1 + \left(\frac{n}{d} - 1\right)z, \ldots, a_d + \left(\frac{n}{d} - 1\right)z\right)\right).$$
Now, let $J'$ be the subset of $\{n_1, \ldots, n_d\}$ to which $J$ corresponds.  Note that $J'$ must be a \emph{proper} subset, else we would have $J = S$ and $\lambda$ would be the trivial character.  The number of characters $\chi_{\lambda,J}^\s$ which satisfy $\omega^d.\chi_{\lambda,J}^\s = \chi_{\lambda,J}^\s$ and for which $J$ corresponds to a fixed $J'$ is therefore equal to 
$$(q - 1)^{d - 1 - |J'|}\left(\frac{n}{d},q - 1\right).$$
Hence, the total number of supersingular characters satisfying $\omega^d.\chi_{\lambda,J}^\s = \chi_{\lambda,J}^\s$ is equal to
\begin{eqnarray*}
-1 + \sum_{J'\subsetneq\{n_1, \ldots, n_d\}}(q - 1)^{d - 1 - |J'|}\left(\frac{n}{d}, q - 1\right) & = & -1 + [d]\left(\frac{n}{d}, q - 1\right)
\end{eqnarray*}
(the $-1$ accounts for the contribution of the trivial character $\chi_{1,\emptyset}^\s$).  

It remains to verify how many of these characters are regular.  Let $\widetilde{\lambda}:T(k)\longrightarrow \fpb^\times$ be a character whose restriction to $T_\s(k)$ is equal to $\lambda$, and let $e$ be a proper divisor of $n$.  Denote by ${\chi}_{\widetilde{\lambda},J}:\cH_\aff\longrightarrow\fpb^\times$ the character of the affine Hecke algebra $\cH_\aff$ associated to $\widetilde{\lambda}$ and $J$ (so that $\chi_{\widetilde{\lambda},J}|_{\cH_\s} = \chi_{\lambda,J}^\s$), and assume $\omega^e.{\chi}_{\widetilde{\lambda},J} = {\chi}_{\widetilde{\lambda},J}$.   This implies in particular that $\omega^{-e}J\omega^e = J$; hence, we obtain $\omega^{-(d,e)}J\omega^{(d,e)} = J$, and the set of such $J$ correspond bijectively to subsets $J'$ of $\{n_1, \ldots, n_{(d,e)}\}$.

We let $\widetilde{\lambda}$ correspond to 
$$(a_1, a_2, \ldots, a_n)\in (\mathbb{Z}/(q - 1)\mathbb{Z})^n$$ 
(lifting the class $((a_1, a_2, \ldots, a_n))\in(\mathbb{Z}/(q-1)\mathbb{Z})^n/\langle(1,1,\ldots, 1)\rangle$ above).  By the above computation, the $n$-tuple corresponding to $\widetilde{\lambda}$ satisfies
\begin{equation}\label{modd}
a_i \equiv a_{\ol{i}} + \left\lfloor\frac{i - 1}{d}\right\rfloor z~(\textnormal{mod}~q - 1),
\end{equation}
where $0< i \leq n$, and where $0 < \overline{i} \leq d$ is congruent to $i$ modulo $d$.  
The condition $\widetilde{\lambda}^{\omega^e} = \widetilde{\lambda}$ implies, in particular, that
$$a_1 \equiv a_{1 + de/(d,e)} \equiv a_1 + \frac{e}{(d,e)}z~(\textnormal{mod}~q - 1),$$
which is equivalent to
$$z \equiv 0~\left(\textnormal{mod}~\frac{q - 1}{\left(\frac{e}{(d,e)}, q - 1\right)}\right).$$

%Equation \eqref{modd} above also has the following consequence.  Fix $0 < i \leq (d,e)$.  For any integer $0\leq m< \frac{n}{(d,e)}$, there exists a unique integer $0\leq j< \frac{d}{(d,e)}$ such that $je \equiv m(d,e)~(\textnormal{mod}~d)$.  This gives
%\begin{eqnarray*}
%a_{i + m(d,e)} & \equiv & a_{\ol{i + m(d,e)}} + \left\lfloor\frac{i + m(d,e) - 1}{d}\right\rfloor z \\ 
% & \equiv &  a_{\ol{i + je}} + \left\lfloor\frac{i + m(d,e) - 1}{d}\right\rfloor z\\
% & \equiv & a_{i + je} + \left\lfloor\frac{i + m(d,e) - 1}{d}\right\rfloor z - \left\lfloor\frac{i + je - 1}{d}\right\rfloor z\\
% & \equiv & a_{i} + \left\lfloor\frac{i + m(d,e) - 1}{d}\right\rfloor z - \left\lfloor\frac{i + je - 1}{d}\right\rfloor z\\
% & \equiv & a_{i} + \left\lfloor\frac{m(d,e) - 1}{d}\right\rfloor z - \left\lfloor\frac{je - 1}{d}\right\rfloor z\quad(\textnormal{mod}~ q - 1).
%\end{eqnarray*}
As above, the character $\widetilde{\lambda}$ is determined by the integers $a_1, \ldots, a_{(d,e)}$ and the element $z$.  Proceeding as above, given a proper subset $J'$ of $\{n_1, \ldots, n_{(d,e)}\}$, we obtain 
$$(q - 1)^{(d,e) - 1 - |J'|}\left(\frac{e}{(d,e)},q - 1\right)$$ 
characters $\chi_{\lambda,J}^\s$ such that the lift ${\chi}_{\widetilde{\lambda},J}$ has an orbit of size dividing $e$, with $J$ corresponding to a fixed $J'$.  Hence, the total number of supersingular characters $\chi_{\lambda,J}^\s$ such that the lift ${\chi}_{\widetilde{\lambda},J}$ has an orbit of size dividing $e$ is 
$$ -1 + [(d,e)]\left(\frac{e}{(d,e)},q - 1\right).$$

By the inclusion-exclusion principle, the number of regular supersingular characters of $\cH_\s$ of orbit size dividing $d$ is
$$-1 + [d]\left(\frac{n}{d}, q - 1\right) - \sum_{j = 1}^{\sigma_0(n) - 1}(-1)^{j + 1}\sum_{\sub{1\leq e_1 < \ldots < e_j < n}{e_i|n}}-1 + [(d, e_1, \ldots, e_j)]\left(\frac{(e_1, \ldots, e_j)}{(d,e_1, \ldots, e_j)}, q - 1\right).$$
Applying Lemma \ref{comb} with $f(e) = -1 + [(d,e)]\left(\frac{e}{(d,e)},q - 1\right)$ and using the fact that \linebreak $\sum_{e|n}\mu\left(\frac{n}{e}\right) = 0$ gives the result.  
\end{proof}

\begin{remark}
 Evaluating the function $g$ at $1$, we obtain
$$g(1) = \sum_{e|n}\mu\left(\frac{n}{e}\right)(e, q - 1).$$
As a function of $n$, the above expression is multiplicative, which implies
$$g(1) = \begin{cases}\varphi(n) & \textnormal{if}~(n,q - 1) = n,\\ 0 & \textnormal{if}~(n,q - 1)\neq n,\end{cases}$$
where $\varphi$ denotes Euler's phi function.  
\end{remark}

\begin{cor}\label{numorbits}
 Let $d$ be a divisor of $n$, and let $h(d)$ denote the number of regular supersingular $L$-packets of $\cH_\s$-modules of size $d$.  We then have
$$h(d) = \frac{1}{d}\sum_{e|d} \mu\left(\frac{d}{e}\right)g(e).$$
\end{cor}

\begin{lem}\label{g(d)neq0}
 Let $d$ be a divisor of $n$.  
\begin{enumerate}
 \item We have $g(d)\neq 0$ if and only if $\left(\frac{n}{d}, q - 1\right) = \frac{n}{d}$.  
 \item If $h(d)\neq 0$, then $\left(\frac{n}{d}, q - 1\right) = \frac{n}{d}$.
\end{enumerate}
\end{lem}

\begin{proof}
By Corollary \ref{numorbits}, it suffices to prove the first claim.  The proof is a tedious (but straightforward) exercise in elementary number theory, and is left to the reader.
\end{proof}

\section{Galois Groups and Projective Galois Representations}\label{galois}

\subsection{Galois Groups}  Let $\gal_{F} := \textrm{Gal}(\ol{F}/F)$ denote the absolute Galois group of $F$, and let $\ii_{F}$ denote the inertia subgroup of elements which act trivially on the residue field $k_{\ol{F}}$.  For any extension $L$ of $F$ contained in $\ol{F}$, we define $\gal_L := \textrm{Gal}(\ol{F}/L)$.  We have $\gal_{F}/\ii_{F} \cong \textrm{Gal}(k_{\overline{F}}/k)\cong \widehat{\mathbb{Z}}$; we denote by $\textnormal{Fr}_q$ a fixed element of $\gal_{F}$ whose image in $\textnormal{Gal}(k_{\overline{F}}/k)$ is the geometric Frobenius element.  Finally, for $m\geq 1$, we let $F_{m}$ denote the unique unramified extension of $F$ of degree $m$ contained in $\ol{F}$.

We fix a compatible system $\{\sqrt[q^m-1]{\varpi}\}_{m\geq 1}$ of $(q^m-1)^{\textnormal{th}}$ roots of $\varpi$, and let $\omega_m:\ii_{F}\longrightarrow \fpb^\times$ denote the character given by 
\begin{equation}\label{omegam}
\omega_m: h\longmapsto \iota\circ\mathfrak{r}_{\ol{F}}\left(\frac{h.\sqrt[q^m-1]{\varpi}}{\sqrt[q^m-1]{\varpi}}\right),
\end{equation}
where $h\in \ii_{F}$ and $\mathfrak{r}_{\ol{F}}:\oo_{\ol{F}}\longrightarrow k_{\ol{F}}$ denotes the reduction modulo the maximal ideal.  Lemma 2.5 of \cite{Br07} shows that the character $\omega_m$ extends to a character of $\gal_{F_{m}}$; we continue to denote by $\omega_m$ the extension which sends the element $\textnormal{Fr}_q^m$ to 1.  Moreover, for $\lambda\in\fpb^\times$, we let $\mu_{m,\lambda}:\gal_{F_m}\longrightarrow \fpb^\times$ denote the unramified character which sends $\textnormal{Fr}_q^m$ to $\lambda$.  Lemma 2.2 of \emph{loc. cit.} implies that every smooth $\fpb$-character of $\gal_{F_m}$ is of the form $\mu_{m,\lambda}\omega_m^r$ with $\lambda\in\fpb^\times$ and $0\leq r < q^m - 1$.  

\subsection{Galois Representations}  

We begin by recalling the classification of irreducible $n$-dimensional mod-$p$ representations of the group $\gal_F$.  Throughout, we assume that $\textnormal{GL}_n(\fpb)$ and $\textnormal{PGL}_n(\fpb)$ are given the discrete topology.  We take \cite{Vig97}, Sections 1.13 and 1.14, and \cite{Be10}, Section 2, as our references.

An element $r$ of $\mathbb{Z}/(q^n - 1)\mathbb{Z}$ is said to be \emph{primitive} if we have 
$$r\not\equiv 0~\left(\textnormal{mod}~\frac{[n]}{[d]}\right)$$
for every proper divisor $d$ of $n$.  The necessary results are summarized in the following proposition.

\begin{propn}\label{galreps}\hfill
\begin{enumerate}[(1)]
\item Any continuous irreducible $n$-dimensional mod-$p$ representation of $\gal_F$ is isomorphic to
$$\textnormal{ind}_{\gal_{F_n}}^{\gal_{F}}(\mu_{n,\lambda}\omega_n^r),$$
where $\lambda\in\fpb^\times$, and $r\in\mathbb{Z}/(q^n - 1)\mathbb{Z}$ is primitive.  
\item We have an isomorphism
$$\textnormal{ind}_{\gal_{F_n}}^{\gal_{F}}(\mu_{n,\lambda}\omega_n^r) \cong \textnormal{ind}_{\gal_{F_n}}^{\gal_{F}}(\mu_{n,\lambda'}\omega_n^{r'})$$
if and only if $\lambda' = \lambda$ and $r' \equiv q^ar~(\textnormal{mod}~q^n - 1)$ for some $a\in \mathbb{Z}$.  
\end{enumerate}
\end{propn}

\begin{proof}
 See Section 1.14 of \cite{Vig97}, or Lemma 2.1.4 and the subsequent remarks in \cite{Be10}.
\end{proof}

We now consider projective representations.  
\begin{defn}
By an \emph{$n$-dimensional mod-$p$ projective Galois representation} we mean a continuous homomorphism from $\gal_F$ to $\textnormal{PGL}_n(\fpb)$.  We say a projective Galois representation is \emph{irreducible} if it does not factor through a proper parabolic subgroup of $\textnormal{PGL}_n(\fpb)$.  Moreover, we say two projective representations $\sigma$ and $\sigma'$ are \emph{equivalent} if there exists an element $m\in\textnormal{PGL}_n(\fpb)$ such that $\sigma(g) = m\sigma'(g)m^{-1}$ for all $g\in \gal_F$.  This equivalence relation will be denoted $\sigma\sim\sigma'$.  
\end{defn}

Given any continuous $n$-dimensional Galois representation $\rho$, we denote by 
$$\llbracket\rho\rrbracket: \gal_F\stackrel{\rho}{\longrightarrow} \textnormal{GL}_n(\fpb) \stackrel{\llbracket-\rrbracket}{\longrightarrow} \textnormal{PGL}_n(\fpb)$$ 
the projective representation obtained as the composition of $\rho$ with the natural quotient map.  The extent to which these representations constitute all projective Galois representations is given by the following theorem.

\begin{thm}\label{tatesthm}
 We have 
$$\textnormal{H}^2(\gal_F,\fpb^\times) = 0.$$
Consequently, every irreducible $n$-dimensional projective Galois representation lifts to a genuine Galois representation, i.e., is of the form $\llbracket\rho\rrbracket$, where $\rho$ is a continuous irreducible $n$-dimensional Galois representation.  
\end{thm}

\begin{proof}
 This follows from (the proof of) Theorem 4 (and its corollary) in \cite{Se77}; one simply uses the decomposition $\fpb^\times \cong \bigoplus_{\ell\neq p} \mathbb{Q}_\ell/\mathbb{Z}_\ell$.
\end{proof}

\begin{lem}\label{lambda=1}
 We have an equivalence 
$$\llbracket\textnormal{ind}_{\gal_{F_n}}^{\gal_F}(\mu_{n,\lambda}\omega_n^r)\rrbracket\sim \llbracket\textnormal{ind}_{\gal_{F_n}}^{\gal_F}(\omega_n^r)\rrbracket.$$
\end{lem}

\begin{proof}
 Let $\sqrt[n]{\lambda}$ denote an $n^{\textnormal{th}}$ root of $\lambda$.  We then have
$$\llbracket\textnormal{ind}_{\gal_{F_n}}^{\gal_F}(\mu_{n,\lambda}\omega_n^r)\rrbracket\sim\llbracket\textnormal{ind}_{\gal_{F_n}}^{\gal_F}(\omega_n^r)\otimes \mu_{1,\sqrt[n]{\lambda}}\rrbracket\sim\llbracket\textnormal{ind}_{\gal_{F_n}}^{\gal_F}(\omega_n^r)\rrbracket.$$\end{proof}

Since we are interested in irreducible projective Galois representations, the above results imply we only need to consider representations of the form $\llbracket\textnormal{ind}_{\gal_{F_n}}^{\gal_F}(\omega_n^r)\rrbracket$ with $r$ primitive.  

\begin{defn}
 Let $\sigma$ be an irreducible projective Galois representation of dimension $n$.  We will say a Galois representation $\rho$ is a \emph{lift of} $\sigma$ if $\rho$ is of the form 
$$\rho = \textnormal{ind}_{\gal_{F_n}}^{\gal_F}(\omega_n^r)$$
and $\llbracket\rho\rrbracket\sim \sigma$.  
\end{defn}

Note that by Theorem \ref{tatesthm} and Lemma \ref{lambda=1}, such a lift always exists.  Moreover, any lift of $\llbracket\textnormal{ind}_{\gal_{F_n}}^{\gal_F}(\omega_n^r)\rrbracket$ is of the form 
$$\textnormal{ind}_{\gal_{F_n}}^{\gal_F}(\omega_n^{q^ar})\otimes\omega_1^m\cong \textnormal{ind}_{\gal_{F_n}}^{\gal_F}(\omega_n^{q^ar + m[n]}) \cong \textnormal{ind}_{\gal_{F_n}}^{\gal_F}(\omega_n^{q^a(r + m[n])}).$$
Hence, the representations $\{\textnormal{ind}_{\gal_{F_n}}^{\gal_F}(\omega_n^{r + m[n]})\}_{0\leq m < q - 1}$ constitute all lifts of $\llbracket\textnormal{ind}_{\gal_{F_n}}^{\gal_F}(\omega_n^r)\rrbracket$, up to isomoprhism (and possibly with repetition).  As the group $\bbZ/(q - 1)\bbZ$ acts transitively on this set by cyclically permuting the $m$ parameter, we see that any irreducible projective Galois representation has $\frac{q - 1}{d'}$ isomorphism classes of lifts, where $d'$ divides $q - 1$.  This discussion also shows we may always choose a lift satisfying $0 \leq r < [n]$.

\begin{propn}\label{rprim}
 Let $d$ be a divisor of $n$ and let $r\in\mathbb{Z}/(q^n - 1)\mathbb{Z}$ with $0\leq r < [n]$.  The number of such $r$ which are primitive and satisfy $q^dr\equiv r~(\textnormal{mod}~[n])$ is equal to 
$$g(d) = \sum_{e|n} \mu\left(\frac{n}{e}\right)[(d,e)]\left(\frac{e}{(d,e)}, q - 1\right).$$
\end{propn}

\begin{proof}
Assume we have $(q^d - 1)r \equiv 0~(\textnormal{mod}~ [n])$; this easily implies $r \equiv s\frac{[n]}{[d]\left(\frac{n}{d}, q - 1\right)}~(\textnormal{mod}~q^n - 1)$ with $0\leq s < [d](\frac{n}{d}, q - 1)$.  %cf. Apostol, Ex. 1.24.
It remains to determine which of these elements are primitive.  Let $e$ be a proper divisor of $n$, and assume 
$$r \equiv s\frac{[n]}{[d]\left(\frac{n}{d}, q - 1\right)} \equiv 0~\left(\textnormal{mod}~\frac{[n]}{[e]}\right);$$  
%We have
%\begin{eqnarray*}
% \left(\frac{[n]}{[d]\left(\frac{n}{d},q - 1\right)}, \frac{[n]}{[e]}\right) & = & \frac{(q^n - 1)}{(q^d - 1)(q^e - 1)\left(\frac{n}{d}, q - 1\right)}\left(q^e - 1, (q^d - 1)\left(\frac{n}{d}, q - 1\right)\right)\\
% & \stackrel{\textnormal{Apostol, Ex. 1.24}}{=} & \frac{(q^n - 1)}{(q^d - 1)(q^e - 1)\left(\frac{n}{d}, q - 1\right)}(q^e - 1, q^d - 1)\left(\frac{q^e - 1}{(q^e - 1, q^d - 1)}, \frac{n}{d}, q - 1\right)\\
% & = & \frac{(q^n - 1)(q^{(d,e)} - 1)}{(q^d - 1)(q^e - 1)\left(\frac{n}{d}, q - 1\right)}\left(\frac{q^e - 1}{q^{(d,e) - 1}}, \frac{n}{d}, q - 1\right)\\
% & = & \frac{(q^n - 1)(q^{(d,e)} - 1)}{(q^d - 1)(q^e - 1)\left(\frac{n}{d}, q - 1\right)}\left(\frac{e}{(d,e)}, \frac{n}{d}, q - 1\right)\\
% & = &  \frac{[n][(d,e)]}{[d][e]\left(\frac{n}{d}, q - 1\right)}\left(\frac{e}{(d,e)}, q - 1\right).
%\end{eqnarray*}
again, one easily checks that this is equivalent to 
$$s \equiv 0~\left(\textnormal{mod}~\frac{[d]\left(\frac{n}{d},q - 1\right)}{[(d,e)]\left(\frac{e}{(d,e)},q - 1\right)}\right).$$
Hence, the number of $0 \leq r < [n]$ satisfying $q^dr \equiv r~(\textnormal{mod}~[n])$ and $r\equiv 0~\left(\textnormal{mod}~\frac{[n]}{[e]}\right)$ is 
$$[(d,e)]\left(\frac{e}{(d,e)}, q - 1\right).$$
We use the inclusion-exclusion principle and Lemma \ref{comb} to conclude (cf. proof of Proposition \ref{orbdivd}).
\end{proof}

\begin{cor}\label{cormind}
 Let $d$ be a divisor of $n$, and let $r\in\mathbb{Z}/(q^n - 1)\mathbb{Z}$ with $0\leq r < [n]$.  The number of such $r$ which are primitive and such that $d$ is the minimal integer satisfying $q^dr \equiv r~(\textnormal{mod}~ [n])$ is
$$\sum_{e|d}\mu\left(\frac{d}{e}\right)g(e).$$
\end{cor}

We define a map from isomorphism classes of irreducible $n$-dimensional projective Galois representations $\sigma$ to $\bbZ_{>0}$ by the formula
$$d_\sigma:=\min\{e\in \bbZ_{>0}: q^er\equiv r~(\textnormal{mod}~[n])\},$$
where $\textnormal{ind}_{\gal_{F_n}}^{\gal_F}(\omega_n^r)$ is a fixed lift of $\sigma$.  One easily checks that this definition is independent of the choice of lift.  

\begin{lem}\label{rtoprojrep}
The integer $d_\sigma$ divides $n$ and $\frac{n}{d_\sigma}$ divides $q - 1$.  Moreover, the projective representation $\sigma$ has exactly $d_\sigma\frac{q - 1}{n}$ isomorphism classes of lifts. 
\end{lem}

\begin{proof}
Let $\textnormal{ind}_{\gal_{F_n}}^{\gal_F}(\omega_n^r)$ be a fixed lift of $\sigma$, and define $\cS_\sigma := \{e\in \bbZ_{>0}:q^er \equiv r~(\textnormal{mod}~[n])\}$; this definition is independent of the choice of lift.  Additionally, the following properties are easily verified:  $n\in \cS_\sigma$, and if $e_1, e_2\in \cS_\sigma$, then $(e_1,e_2)\in \cS_\sigma$.  The properties above and minimality of $d_\sigma$ imply that if $e\in \cS_\sigma$, then $d_\sigma\leq (d_\sigma,e)\leq d_\sigma$, so $d_\sigma$ divides $e$.  In particular, $d_\sigma$ divides $n$, and we have $\cS_\sigma = d_\sigma\bbZ_{>0}$.

Now, by definition of $\cS_\sigma$, we have $q^{d_\sigma}r \equiv r + m[n]~(\textnormal{mod}~q^n - 1)$ for some $m\in \bbZ/(q - 1)\bbZ$.  Applying this equation recursively yields
\begin{equation}\label{rtoprojrepform}
q^{ad_\sigma}r \equiv r + am[n]~(\textnormal{mod}~q^n - 1),
\end{equation}
for $a\in \bbZ$.  Taking $a = q - 1$ in equation \eqref{rtoprojrepform} yields
$$q^{(q - 1)d_\sigma}r \equiv r~(\textnormal{mod}~q^n - 1),$$
which implies that $n$ divides $(q - 1)d_\sigma$ by primitivity of $r$.

Taking $a = \frac{n}{d_\sigma}$ in equation \eqref{rtoprojrepform} gives
$$\frac{n}{d_\sigma}m[n]\equiv 0~(\textnormal{mod}~q^n - 1),$$
which is equivalent to $m \equiv 0~\left(\textnormal{mod}~\frac{d_\sigma(q - 1)}{n}\right)$ (here we use that $\frac{n}{d_\sigma}$ divides $q - 1$ from above).  Let us write $m\equiv m'\frac{d_\sigma(q - 1)}{n}$, with $m'\in \bbZ/(n/d_\sigma)\bbZ$, so that equation \eqref{rtoprojrepform} becomes
$$q^{ad_\sigma}r \equiv r + am'\frac{d_\sigma(q - 1)}{n}[n]~(\textnormal{mod}~q^n - 1).$$
We claim $m'$ and $\frac{n}{d_\sigma}$ are relatively prime.  If not, then there would exist $0< a < \frac{n}{d_\sigma}$ such that $am'\equiv 0~(\textnormal{mod}~\frac{n}{d_\sigma})$, which implies $q^{ad_\sigma}r\equiv r~\left(\textnormal{mod}~q^n - 1\right)$, contradicting the primitivity of $r$.  Therefore, we can choose $a'$ such that $a'm' \equiv 1~\left(\textnormal{mod}~\frac{n}{d_\sigma}\right)$, which gives
$$q^{a'd_\sigma}r \equiv r + \frac{d_\sigma(q - 1)}{n}[n]~(\textnormal{mod}~q^n - 1).$$
In particular, this shows that
\begin{equation}\label{rtoprojrepcong}
\textnormal{ind}_{\gal_{F_n}}^{\gal_F}(\omega_n^r)\otimes \omega_1^{d_\sigma(q - 1)/n}\cong \textnormal{ind}_{\gal_{F_n}}^{\gal_F}(\omega_n^r).
\end{equation}

The discussion preceding Proposition \ref{rprim} shows that there exists an integer $d'$ dividing $q - 1$ such that 
$$\textnormal{ind}_{\gal_{F_n}}^{\gal_F}(\omega_n^r),~ \textnormal{ind}_{\gal_{F_n}}^{\gal_F}(\omega_n^r)\otimes\omega_1,\ldots,~ \textnormal{ind}_{\gal_{F_n}}^{\gal_F}(\omega_n^r)\otimes\omega_1^{((q - 1)/d') - 1}$$
are a full and pairwise nonisomorphic set of representatives of lifts of $\llbracket\textnormal{ind}_{\gal_{F_n}}^{\gal_F}(\omega_n^r)\rrbracket$.  This implies $\textnormal{ind}_{\gal_{F_n}}^{\gal_F}(\omega_n^r)\otimes\omega_1^{(q - 1)/d'}\cong \textnormal{ind}_{\gal_{F_n}}^{\gal_F}(\omega_n^r)$, meaning
$$q^{bd_\sigma}r\equiv r + \frac{q - 1}{d'}[n]~(\textnormal{mod}~q^n - 1)$$
for some $b\in \bbZ$ (note that the exponent of $q$ on the left-hand side must be an element of $\cS_\sigma$).  Proceeding as above, we iterate this relation $\frac{n}{d_\sigma}$ times to obtain
$$\frac{n}{d_\sigma}\frac{q - 1}{d'}[n]\equiv 0~(\textnormal{mod}~q^n - 1),$$
which is equivalent to saying that $d_\sigma d'$ divides $n$.  On the other hand, the definition of $d'$ and equation \eqref{rtoprojrepcong} show that $\frac{q - 1}{d'}$ divides $\frac{d_\sigma(q - 1)}{n}$.  These two facts show that $d' = \frac{n}{d_\sigma}$.  
\end{proof}

\begin{cor}\label{numprojreps}
 Let $d$ be a divisor of $n$.  The number of isomorphism classes of irreducible projective Galois representations $\sigma$ of dimension $n$ for which $d_\sigma = d$ is equal to 
$$h(d) = \frac{1}{d}\sum_{e|d}\mu\left(\frac{d}{e}\right)g(e).$$
\end{cor}

\begin{proof}
Let $\textnormal{pr}$ denote the surjective map from the set of primitive integers $r$ such that $0\leq r < [n]$, to the set of isomorphism classes of irreducible $n$-dimensional projective Galois representations given by $\textnormal{pr}(r) = \llbracket\textnormal{ind}_{\gal_{F_n}}^{\gal_F}(\omega_n^r)\rrbracket$.  The sizes of the fibers of the map $r\longmapsto d_{\textnormal{pr}(r)}$ are given by Corollary \ref{cormind}.  Furthermore, by definition of the integer $d_\sigma$, %and the fact that $\llbracket\textnormal{ind}_{\gal_{F_n}}^{\gal_F}(\omega_n^r)\rrbracket \sim \llbracket\textnormal{ind}_{\gal_{F_n}}^{\gal_F}(\omega_n^{r'})\rrbracket \Longleftrightarrow r' \equiv q^ar + m[n]$
the fiber of the map $\textnormal{pr}$ over the representation $\sigma$ has size $d_\sigma$.  Combining these two facts gives the corollary.  
\end{proof}

\begin{cor}\label{finalcor}
 Let $d$ be a divisor of $n$.  The number of regular supersingular $L$-packets of $\cH_\s$-modules of size $d$ is equal to the number of (isomorphism classes of) irreducible projective Galois representations $\sigma$ of dimension $n$ for which $d_\sigma = d$.  In particular, the number of regular supersingular $L$-packets of $\cH_\s$-modules is equal to the number of (isomorphism classes of) irreducible projective Galois representations of dimension $n$.  
\end{cor}

\begin{proof}
 This follows from Corollaries \ref{numorbits} and \ref{numprojreps}, and Lemma \ref{rtoprojrep}.  
\end{proof}

\begin{remark}
Lemma \ref{rtoprojrep} shows that the condition ``$d_\sigma = d$'' is related to the number of lifts of $\sigma$.
\end{remark}

\subsection{Comparison with Gro\ss e-Kl\"onne's Functor}\label{gk}

Gro\ss e-Kl\"onne has recently constructed a functor from the category of finite-length right $\cH_\bullet$-modules to the category of \'etale $(\varphi^r,\Gamma_0)$-modules (\cite{GK13}).  When applied to the group $\textnormal{GL}_n(\qp)$, his construction, composed with Fontaine's equivalence of categories, yields a bijection between isomorphism classes of (absolutely) simple, supersingular right $\cH$-modules of dimension $n$ and isomorphism classes of (absolutely) irreducible $n$-dimensional mod-$p$ representations of $\textnormal{Gal}(\overline{\mathbb{Q}}_p/\qp)$.  We now analyze these $(\varphi^r,\Gamma_0)$-modules for $\textnormal{SL}_n(\qp)$.  For this subsection only, we adhere to the notation of \emph{loc. cit.}; the reader should consult that article for precise statements and definitions.

We take $F = \qp, \oo = \zp$, with residue field $\fp$, and uniformizer $\varpi = p$.  We let $k$ denote the residue field in a fixed (sufficiently large) finite extension of $\qp$.  Recall that we have identified the apartments $A$ and $A_\s$, and we let $C$ denote the chamber in $A_\s$ corresponding the the Iwahori subgroup $I_\s$.  We choose a semiinfinite chamber gallery in $A_\s$ by setting, for $i\geq 0$,
$$C^{(i)} := (n_{n - 1}^{-1}\omega)^i.C,$$
and note that the action on $A_\s$ of
$$(n_{n - 1}^{-1}\omega)^n = \underbrace{n_{n - 1}^{-1}\omega\cdots n_{n - 1}^{-1}\omega}_{n~\textnormal{times}} = \omega^nn_{n - 1}^{-1}n_{n - 2}^{-1}\cdots n_1^{-1}n_0^{-1}$$
is the same as the action of 
$$\phi := n_{n - 1}^{-1}n_{n - 2}^{-1}\cdots n_1^{-1}n_0^{-1} \in G_\s.$$
We have $\phi.C^{(i)} = C^{(i + n)}$ by definition.  The choice of a chamber gallery and such an element $\phi$ provides us with a half tree $Y_\s\subset X_\s$ and a simplicial isomorphism between $Y_\s$ and ``the half tree of $\textnormal{PGL}_2(\qp)$'' (cf. \emph{loc. cit.}, Section 3).

To every simple supersingular $\cH_\s$-module $\chi_{\lambda,J}^\s$ we associate an $n$-tuple of integers 
$$(k_1,\ldots, k_{n - 1}, k_n)$$ 
as follows (cf. \emph{loc. cit.}, Section 8).  For $0\leq i \leq n - 1$, we let $0 \leq k_{i + 1} \leq p - 1$ satisfy
$$\lambda\circ\alpha_{n_i}^\vee(x^{-1}) = x^{k_{i + 1}}.$$
If $\lambda\circ \alpha_{n_i}^\vee$ is not the trivial character, then $k_{i + 1}$ is uniquely determined.  The condition of $\lambda\circ\alpha_{n_i}^\vee$ being equal to the trivial character is equivalent to $n_i\in S_\lambda$; in this case, we set $k_{i + 1} = p - 1$ if $n_i\in J$ and $k_{i + 1} = 0$ otherwise.

Tracing through the construction of \emph{loc. cit.}, we arrive at the following proposition.

\begin{propn}\label{phingamma0}
Let $\chi_{\textnormal{cyc}}:\Gamma\stackrel{\sim}{\longrightarrow}\zp^\times$ denote the cyclotomic character, and let $\Gamma_0 = \chi_{\textnormal{cyc}}^{-1}(1 + p\zp)$.  The \'etale $(\varphi^n,\Gamma_0)$-module $\mathbf{D}_{\lambda,J}$ associated to $\chi_{\lambda,J}^\s$ is one-dimensional over $k_\mathcal{E} = k((t))$, spanned by a vector $\vec{e}$, with actions given by
\begin{eqnarray*}
\varphi^n(\vec{e}) & = & (-1)^n\left(\prod_{j = 1}^{n}k_j!\right)^{-1}t^{-\sum_{j = 0}^{n - 1}(p - 1 - k_{n - j})p^j}\vec{e},\\
\gamma(\vec{e}) & = & \left(\frac{t}{(1 + t)^{\chi_{\textnormal{cyc}}(\gamma)} - 1}\right)^{\frac{1}{p^n - 1}\sum_{j = 0}^{n - 1}(p - 1 - k_{n - j})p^j}\vec{e},
\end{eqnarray*}
where $\gamma\in \Gamma_0$.  In particular, we see that distinct supersingular characters $\chi_{\lambda,J}^\s$ give rise to distinct $(\varphi^n,\Gamma_0)$-modules.  
\end{propn}

\begin{proof}
This is a straightforward computation using \cite{GK13}.  
\end{proof}

\begin{remark}
 The construction of $\mathbf{D}_{\lambda,J}$ depends on the choice of chamber gallery $C^{(0)}, C^{(1)}, C^{(2)}, \ldots$ and the element $\phi$.  
\end{remark}

We can push the construction of Proposition \ref{phingamma0} a bit further.  Given a one-dimensional $(\varphi^n,\Gamma_0)$-module $\mathbf{D}_{\lambda,J}$ as above, we construct an $n$-dimensional \'etale $(\varphi,\Gamma_0)$-module as follows.  Let $\widetilde{\mathbf{D}_{\lambda,J}}$ denote the $k_{\mathcal{E}}$ vector space spanned by $\{\vec{e}_0,\ldots, \vec{e}_{n - 1}\}$, with actions given by
\begin{eqnarray*}
 \varphi(\vec{e}_i) & = & \begin{cases}\vec{e}_{i + 1} & \textnormal{if}~ 0\leq i < n - 1,\\
 \displaystyle{(-1)^n\left(\prod_{j = 1}^{n}k_j!\right)^{-1}}t^{-\sum_{j = 0}^{n - 1}(p - 1 - k_{n - j})p^j}\vec{e}_0 & \textnormal{if}~ i = n - 1, \end{cases}\\
 \gamma(\vec{e}_i) & = & \left(\frac{t}{(1 + t)^{\chi_{\textnormal{cyc}}(\gamma)} - 1}\right)^{\frac{p^i}{p^n - 1}\sum_{j = 0}^{n - 1}(p - 1 - k_{n - j})p^j}\vec{e}_i,
\end{eqnarray*}
where $\gamma\in \Gamma_0$.  It is clear that $\widetilde{\mathbf{D}_{\lambda,J}}$ is isomorphic to $\widetilde{\mathbf{D}_{\lambda^{\omega^{-1}},\omega J\omega^{-1}}}$.  %Set $\vec{f}_0 = t^{p - 1 - k_1}\vec{e}_1$, and define $\vec{f}_i = \varphi^i(\vec{f}_0)$.  Then everything goes through as claimed.
Moreover, we have
$$\widetilde{\mathbf{D}_{\lambda,J}} \cong \bigoplus_{i = 0}^{n - 1}\mathbf{D}_{\lambda^{\omega^{-i}},\omega^{i}J\omega^{-i}}$$ 
as $(\varphi^n,\Gamma_0)$-modules, the isomorphism given by sending $k_{\mathcal{E}}.\vec{e}_i$ to $\mathbf{D}_{\lambda^{\omega^{-i}},\omega^{i}J\omega^{-i}}$.  The discussion of Section 2.2 of \cite{Be10} shows that we may (nonuniquely) extend the action of $\Gamma_0$ to $\Gamma$, so that we obtain a bona fide $(\varphi,\Gamma)$-module.  By abuse of notation, we shall also denote this module by $\widetilde{\mathbf{D}_{\lambda,J}}$.

We may now relate simple supersingular $\cH_\s$-modules and projective Galois representations more precisely.  Let $\chi_{\lambda,J}^\s$ be as before, and let $(k_1,\ldots k_{n - 1}, k_n)$ denote the associated $n$-tuple.  We define the rational number $r$ (depending on $\lambda$ and $J$) by
\begin{equation}\label{defofr}
r := \frac{1}{p - 1}\sum_{j = 0}^{n - 1}(p - 1 - k_{n - j})p^j. 
\end{equation}
\noindent By Theorems 8.5, 8.6(c) and 8.7 of \cite{GK13}, $r$ is in fact an integer, and $\chi_{\lambda,J}^\s$ is regular if and only if $r$ is primitive.  

Denote by
$$\mathbf{D}\longmapsto W(\mathbf{D})$$
Fontaine's equivalence of categories, from the category of \'etale $(\varphi,\Gamma)$-modules over $k_{\mathcal{E}}$ to the category of finite-dimensional representations of $\gal_{\qp}$ over $k$ (see \cite{Fon90} for more details).  Applying this functor to the $(\varphi,\Gamma)$-module $\widetilde{\mathbf{D}_{\lambda,J}}$ and using the computations contained in Section 2.2 of \cite{Be10}, we obtain
$$W(\widetilde{\mathbf{D}_{\lambda,J}}) = \textnormal{ind}_{\gal_{\mathbb{Q}_{p^n}}}^{\gal_{\qp}}(\omega_n^r)\otimes \mu_{1,\beta}\omega_1^s$$
for some $0\leq s < p - 1, \beta\in \fpb^\times$, and $r$ as in equation \eqref{defofr}.  Here $\mathbb{Q}_{p^n}$ denotes the unique unramified extension of $\qp$ of degree $n$ contained in a fixed algebraic closure $\ol{\mathbb{Q}}_p$.  Precomposing this with Gro\ss e-Kl\"onne's functor (and the lifting map $\mathbf{D}_{\lambda,J}\longmapsto\widetilde{\mathbf{D}_{\lambda,J}}$) and postcomposing with the projectivization functor, we get a map $\mathcal{W}$ from the set of simple supersingular $\cH_\s$-modules to the category of projective Galois representations.  Explicitly, it is given by
$$\mathcal{W}(\chi_{\lambda,J}^\s) = \llbracket W(\widetilde{\mathbf{D}_{\lambda,J}})\rrbracket \sim \llbracket\textnormal{ind}_{\gal_{\mathbb{Q}_{p^n}}}^{\gal_{\qp}}(\omega_n^r)\rrbracket.$$

There are several immediate consequences of this definition.  Firstly, we see that the isomorphism class of $\mathcal{W}(\chi_{\lambda,J}^\s)$ is independent of the choice of action of $\Gamma$ on the $(\varphi,\Gamma_0)$-module $\widetilde{\mathbf{D}_{\lambda,J}}$, so that the map $\mathcal{W}$ is well-defined.  Secondly, Proposition \ref{galreps} and Theorem \ref{tatesthm} show that $\chi_{\lambda,J}^\s$ is regular if and only if $\mathcal{W}(\chi_{\lambda,J}^\s)$ is irreducible.  Finally, it is clear that two simple supersingular $\cH_\s$-modules in the same orbit under $\omega^{\mathbb{Z}}$ will yield isomorphic projective Galois representations under $\mathcal{W}$, so we may view $\mathcal{W}$ as a map defined on supersingular $L$-packets of $\cH_\s$-modules.

\begin{propn}\label{orblifts}
 Let $\chi_{\lambda,J}^\s$ be a simple, regular supersingular $\cH_\s$-module, and let $d$ be a divisor of $n$.  Then the orbit of $\chi_{\lambda,J}^\s$ under $\omega^{\mathbb{Z}}$ has size $d$ if and only if $d_{\mathcal{W}(\chi_{\lambda,J}^\s)} = d$.
\end{propn}

\begin{proof}
 Let $(k_1,\ldots, k_{n - 1}, k_n)$ denote the $n$-tuple of integers associated to $\chi_{\lambda,J}^\s$, $r$ the integer given by equation \eqref{defofr}, and $d$ a divisor of $n$.  The group $\omega^\mathbb{Z}$ acts on the set of $n$-tuples by 
$$\omega.(k_1,\ldots, k_{n - 1}, k_n) := (k_n, k_1,\ldots, k_{n - 1}).$$
One easily verifies the following equivalences:
\begin{eqnarray*}
\omega^d.\chi_{\lambda, J}^\s = \chi_{\lambda, J}^\s & \Longleftrightarrow & \omega^d.(k_1,\ldots, k_{n - 1}, k_n) = (k_1, \ldots, k_{n - 1}, k_n)\\ 
 & \Longleftrightarrow & r = \frac{1}{p - 1}\frac{p^n - 1}{p^d - 1}\sum_{j = 0}^{d - 1}(p - 1 - k_{n - j})p^j\\
 & \Longleftrightarrow & p^dr =  r + [n]\sum_{j = 0}^{d - 1}(p - 1 - k_{n - j})p^j\\
 & \Longleftrightarrow & p^dr \equiv r~(\textnormal{mod}~[n]).
\end{eqnarray*}
By the definition of $d_\sigma$, we are done.  
\end{proof}

\begin{cor}\label{Wbij}
 The map $\mathcal{W}$ realizes the numerical bijection of Corollary \ref{finalcor}.  More precisely, $\mathcal{W}$ induces a bijection between regular supersingular $L$-packets of $\cH_\s$-modules of size $d$ and irreducible projective representations $\sigma$ of $\gal_{\qp}$ of dimension $n$ for which $d_\sigma = d$.  
 \end{cor}

\begin{proof}
 Given a irreducible projective representation $\llbracket \textnormal{ind}_{\gal_{\mathbb{Q}_{p^n}}}^{\gal_{\qp}}(\omega_n^r)\rrbracket$, Theorem 8.7 of \cite{GK13} shows how to construct a simple, regular supersingular $\cH_\s$-module $\chi_{\lambda,J}^\s$ such that $\mathcal{W}(\chi_{\lambda,J}^\s) \sim \llbracket \textnormal{ind}_{\gal_{\mathbb{Q}_{p^n}}}^{\gal_{\qp}}(\omega_n^r)\rrbracket$.  Hence, the map $\mathcal{W}$ from regular supersingular $L$-packets of size $d$ to irreducible projective Galois representations with $d_\sigma = d$ is surjective.  By Corollary \ref{finalcor}, these two sets have the same size, and $\mathcal{W}$ must be injective as well.
\end{proof}

\begin{cor}\label{compat}
 Let $\MM$ be an $\cH$-module.  We let $\MM\longmapsto\mathcal{GK}(\MM)$ denote the functor from the category of finite-length $\cH$-modules over $k$ to the category of continuous $\gal_{\qp}$-representations over $k$ constructed in Section 8 of \cite{GK13}, and let $\MM\longmapsto\textnormal{JH}(\MM|_{\cH_{\s}})$ denote the functor obtained by taking the Jordan-H\"older constituents of the $\cH_\s$-module $\MM|_{\cH_\s}$ (without multiplicity).  The following diagram of sets is commutative (where we consider all objects up to isomorphism and over $k$), and the horizontal arrows are bijections:

\centerline{
\xymatrix{
\left\{\txt{ \textnormal{absolutely simple,}\\ \textnormal{supersingular}\\  $n$\textnormal{-dimensional}~ $\cH$\textnormal{-modules}}\right\} \ar[r]^<<<<<{\mathcal{GK}(-)}\ar[d]_{\textnormal{JH}(-|_{\cH_\s})} & \left\{\txt{  \textnormal{absolutely irreducible}\\  $n$\textnormal{-dimensional}\\ $\gal_{\qp}$\textnormal{-representations}}\right\} \ar[d]^{\llbracket - \rrbracket}\\
\left\{\txt{ \textnormal{regular, supersingular}\\$L$\textnormal{-packets of}\\  $\cH_\s$\textnormal{-modules}}\right\} \ar[r]_<<<<<<{\mathcal{W}(-)} & \left\{\txt{  \textnormal{absolutely irreducible}\\  $n$\textnormal{-dimensional projective}\\ $\gal_{\qp}$\textnormal{-representations}}\right\}
}
}
\end{cor}

\begin{proof}
 This follows from the explicit calculation of Theorem 8.5 in \cite{GK13}, and the comments preceding Proposition \ref{orblifts}.  
\end{proof}

\end{document}